\documentclass{amsart}
\usepackage{color, graphicx, enumerate, amssymb}
\usepackage{hyperref}
\usepackage{epsfig,wrapfig}
\usepackage{pxfonts}
\usepackage{graphicx}

\usepackage{eucal}
\usepackage[utf8]{inputenc}
\usepackage{ulem}

\usepackage{amsmath}
\usepackage{geometry}
\usepackage{comment}
\usepackage{enumitem}
\usepackage{thmtools}
\usepackage{cleveref}
\newtheorem{theorem}{Theorem}[section] 
\newtheorem{lemma}[theorem]{Lemma}
\newtheorem{corollary}[theorem]{Corollary}
\newtheorem{definition}{Definition}[section]
\newtheorem{proposition}[theorem]{Proposition}
\newtheorem{remark}[theorem]{Remark}

\newcommand{\supp}{\operatorname{supp}}
\newcommand{\dist}{{\rm dist}}

\newcommand{\rW}{{\mathrm w}}
\newcommand{\bR}{{\mathbb R}}

\newcommand{\cH}{{\mathcal H}}
\newcommand{\cN}{{\mathcal N}}
\newcommand{\cV}{{\mathcal V}}
\newcommand{\cL}{{\mathcal L}}

\title[Non-minimizing Axially Symmetric  Cavity Flow]{Non-minimizing Axially Symmetric Cavity Flow }

\author[M. Bayrami]{Masoud Bayrami}
\address[M. Bayrami]{School of Mathematics, Statistics and Computer Science, College of Science, University of Tehran.}
\email{aminlouee@ut.ac.ir}
\address[M. Bayrami]{School of Mathematics, Institute for Research in Fundamental Sciences (IPM), P.O. Box: 19395-5746, Tehran, Iran.}
\email{aminlouee@ipm.ir}

\author[M. Fotouhi]{Morteza Fotouhi}
\address[M. Fotouhi]{Department of Mathematical Sciences, Sharif University of Technology, P.O. Box: 11365-9415, Tehran, Iran}
\email{fotouhi@sharif.edu}

\author[P. Vosooqnejad]{Parisa Vosooqnejad}
\address[P. Vosooqnejad]{Department of Mathematical Sciences, Sharif University of Technology, P.O. Box: 11365-9415, Tehran, Iran}
\email{par.vosooq@sharif.edu}

\date{\today}

\begin{document}

\begin{abstract}
Cavity flow problems in two dimensions, as well as in the axially symmetric three-dimensional case, have been extensively studied in the literature from a qualitative perspective. 
While numerous results exist concerning minimizers or stable solutions-particularly regarding the regularity of the free boundary and the analysis of singularities—much less is known about the critical points of the corresponding energy functional.  
In this paper, we focus on investigating the properties of such critical points in the axially symmetric cavity flow problem with a free boundary, in relation to the known variational solutions.
Moreover, our approach extends naturally to the case of jet flow problems.
\end{abstract}



\maketitle

\tableofcontents

\section{Introduction}
\subsection{Overview of the Mathematical Framework}

The classical theory of jets and cavities is initially focused on studying incompressible inviscid two-dimensional (ideal) flows; see, e.g., \cite{MR88230, MR119655, MR204008} and references therein. 
Related problems, such as those arising in the theory of steady two-dimensional free-surface water waves, have also been the subject of extensive research. For overviews of these developments, we refer to \cite{zbMATH05964521, zbMATH05800625, zbMATH01127680}; and also \cite{kriventsov2025minmaxvariationalapproachexistence}, which presents a framework that accommodates both a variational approach to various fluid equilibrium problems and the construction of min–max solutions to Bernoulli-type free boundary problems.

A significant development in the study of jets and cavities for compressible fluids using novel variational approaches began with the seminal works of Alt, Caffarelli, and Friedman \cite{MR752578, zbMATH00522487, alt1985compressible, caffarelli1982axially}. 
These variational problems, have been studied within a geometric measure-theoretic framework, employ geometric and blow-up arguments as well as novel monotonicity formulas. For further details, see for instance \cite{alt1983axially,  brock1993axiallyI, brock1993axiallyII, cheng2020free, du2023free1, du2023free2, du2024regularity, zbMATH03982826, zbMATH01844201, zbMATH01080663,  zbMATH05555290, zbMATH05963664, zbMATH06156243, zbMATH06327034, zbMATH02190600, zbMATH04107321, zbMATH04107322, zbMATH05934884}. 

In this paper, we build upon this line of research by focusing not on energy minimizers, but on \textit{non-minimizing critical points}, a class of solutions that is comparatively less understood yet of substantial physical and mathematical interest. 
Our investigation is motivated in part by recent advances in the understanding of variational limits of singular perturbation problems, particularly the work of Kriventsov and Weiss \cite{kriventsov2023rectifiability}, who developed a refined notion of variational solutions that lies between the classical weak and viscosity solution frameworks.

The cavity flow problem we consider arises in both two-dimensional and axially symmetric three-dimensional configurations and involves a free boundary separating a region occupied by fluid from a cavity or vacuum.  
While numerous results exist concerning minimizers or stable solutions—particularly regarding free boundary regularity and singularity analysis—much less is known about the critical points (specifically saddle points) of the corresponding energy functional. 

In the axisymmetric setting, the natural variational functional governing steady ideal incompressible flows is given by:
\begin{equation}
\label{singular-ACF}
J_{\lambda}(\psi) = \int_{\Omega} \left( \left|\dfrac{\nabla \psi}{y} \right|^{2}+ \lambda^{2} \chi_{\{\psi>0\}} \right)y \, dxdy,
\end{equation}
where $ \Omega $ lies in the half-plane $ \{(x,y) \, : \, y \geq 0\} $, $\lambda > 0$ is the cavity speed and $\psi$ is the stream function.

However, the functional \eqref{singular-ACF} is not convex, and it is evident that there exist critical points that are not minimizers. 
In fact, such saddle points are not merely mathematical artifacts but can correspond to physically relevant configurations, particularly in time-dependent or perturbed problems such as those appearing in combustion theory or free surface flows with instabilities. See \cite{MR1989835} for a detailed discussion in the context of parabolic singular perturbations. 
More precisely, in \cite{MR1989835}, Weiss showed that variational solutions naturally arise as limits of more regular solutions in the study of singular perturbations of semilinear equations. An important observation made there is that, although inner-variation solutions are stable under taking limits, this stability holds only in a generalized sense. Specifically, the limit $I$ of $\chi_{\{\psi_k>0\}}$, which forms part of the energy functional, is indeed the characteristic function of a set, but it does not necessarily coincide with the positivity set of the limiting function $\{\psi_0>0\}$. Instead, only the inclusion $\chi_{\{\psi_0>0\}} \leq I$ is guaranteed, which motivates a relaxed definition of variational solutions.

As highlighted in \cite{kriventsov2023rectifiability}, classical notions of weak solutions, as introduced by Alt and Caffarelli in \cite{caffarelli1981existence}, are often too restrictive to capture limits of approximating sequences arising from singular perturbations or viscous regularization. 
Conversely, the class of viscosity solutions, although more robust under limits, is too broad to permit fine structural or regularity results. 
The intermediate class of \textit{variational solutions}, as developed in \cite{MR1989835, kriventsov2023rectifiability}, provides a more natural and analytically tractable setting for studying such limit configurations.

Motivated by this framework, we study the Bernoulli-type free boundary problem associated with variational solutions of the functional \eqref{singular-ACF}:
\begin{equation}
\label{OUR-problem}
\begin{cases}  
\mathcal{L}\psi:=\mathrm{div\,} \left(\dfrac{1}{y} \nabla \psi \right) = 0, \qquad & \text{in} \ \{\psi > 0\} \cap \Omega, \\  
\psi = 0, \qquad & \text{on} \ \partial \{\psi > 0\} \cap \Omega, \\  
\dfrac{1}{y} \dfrac{\partial \psi}{\partial \nu} = \lambda, \qquad & \text{on} \ \partial \{\psi > 0\} \cap \Omega,  
\end{cases}
\end{equation}
which describes a steady axisymmetric incompressible ideal flow in three dimensions, with the $x$-axis serving as the axis of symmetry. 
The differential operator \( \mathcal{L} \) is singular at \( y = 0 \), reflecting the degeneracy introduced by axial symmetry in cylindrical coordinates. 
This dynamical boundary condition which equivalently can be represented as $| \nabla \psi |= \lambda y $ on the free boundary $\partial \{\psi > 0\} \cap \Omega$ indicates that the flow speed scales linearly with the radial distance from the axis, and degenerates as \( y \to 0 \).
Points on the free boundary where the gradient of the stream function vanishes are referred to as \textit{degenerate points}, such as those lying on the symmetry axis $y=0$. For our analysis, we also distinguish free boundary points in the region $y>0$ that are not regular free boundary points, which we call \textit{singular points}. The set of such points is denoted by $\Sigma_{H}$ and is precisely defined in \eqref{sigma_H}.

The focus of this paper is the qualitative and structural analysis of non-minimizing critical points of the functional \eqref{singular-ACF} near the free boundary. 
While our analysis includes both degenerate and nondegenerate points, we place particular emphasis on the singular points lying in the domain \(\{ y > 0 \}\). 
Understanding the structure of the free boundary near such points is crucial for capturing the qualitative behavior of the flow and the associated variational solutions.

\subsection{Notations}
We use \( X \cdot Y \) to denote the Euclidean inner product in \( \mathbb{R}^2 \), and \( |X| \) to denote the Euclidean norm. 
The Euclidean distance from a point \( X \in \mathbb{R}^2 \) to a set \( L \subset \mathbb{R}^2 \) is written as \( d(X, L) \).
The upper half-plane is denoted by \( \mathbb{R}^2_+ := \{ X = (x, y) \in \mathbb{R}^2 \, : \, y \geq 0 \} \), and the standard Lebesgue measure on \( \mathbb{R}^2 \) is written as \( dX = dx dy \). 
When the context is clear, we may omit the measure of integration notation for brevity.

The open ball of radius \( r \) centered at a point \( X_0 \in \mathbb{R}^2 \) is denoted by \( B_r(X_0) := \{ X \in \mathbb{R}^2 \, : \, |X - X_0| < r \} \). 
The intersection of this ball with the upper half-plane is denoted by \( B_r^+(X_0) := B_r(X_0) \cap \mathbb{R}^2_+ \).
We refer to the \textit{circular boundary} of this set as
\[
\partial B_r^{+}(X_0) := \{ X = (x, y) \in \mathbb{R}^2 \, : \, y \geq 0 \text{ and } |X - X_0| = r \}.
\]
Note that \( \partial B_r^+(X_0) \) is not the topological boundary of \( B_r^+(A) \), and \( B_r^+(A) \) is not necessarily a true half-ball. 
When the center is omitted, it is understood to be the origin; thus we write  \( B_r := B_r(0) \) and \( B_r^+ := B_r^+(0) \). 

We assume \( \Omega \subseteq \mathbb{R}^2_+ \) is an open set. The boundary \( \partial \Omega \) refers to the portion of the boundary lying strictly in the upper half-plane and is distinct from the full topological boundary \( \overline{\Omega} \subseteq \mathbb{R}^2 \).
The volume (Lebesgue measure) of a set  \( A \subseteq \mathbb{R}^2 \) is denoted by \( |A| \), representing the two-dimensional Lebesgue measure \( \mathcal{H}^2 \).  
In addition, we denote by  $|A|_{\rW}$ the weighted volume of \( A \) with respect to the measure $y\,dX$, defined by $|A|_{\rW} := \int_A y \, dX$.

To define the solutions of the problem in an appropriate sense, we require certain weighted Sobolev spaces, denoted as \( W^{1,2}_{\rW}(E) \).

\begin{definition}
Let \( E \subseteq \mathbb{R}^{2}_{+} \) be an open set. The weighted Lebesgue space \( L^{2}_{\rW}(E) \) and its local version \( L^{2}_{\rW,\mathrm{loc}}(E) \) are defined as follows:
\[
L^{2}_{\rW}(E) := \left\{ g : E \to \mathbb{R} \, : \, g \text{ is measurable and } \int_{E} \frac{1}{y} |g|^{2} \, dX < +\infty \right\},
\]
and
\[
L^{2}_{\rW,\mathrm{loc}}(E) := \left\{ g \in L^{2}_{\rW}(K) \text{ for every compact set } K \subseteq E \right\}.
\]
The corresponding norm on \( L^{2}_{\rW}(E) \) is
\[
\| g \|_{L^{2}_{\rW}(E)} := \left( \int_{E} \frac{1}{y} |g|^{2} \, dX \right)^{1/2}.
\]
The weighted Sobolev space \( W^{1,2}_{\rW}(E) \) and its local version \( W^{1,2}_{\rW,\mathrm{loc}}(E) \) are given by:
\[
W^{1,2}_{\rW}(E) := \left\{ g \in L^{2}_{\rW}(E) \, : \, \partial_x g, \partial_y g \in L^{2}_{\rW}(E) \right\},
\]
and
\[
W^{1,2}_{\rW,\mathrm{loc}}(E) := \left\{ g \in L^{2}_{\rW,\mathrm{loc}}(E) \, : \, \partial_x g, \partial_y g \in L^{2}_{\rW,\mathrm{loc}}(E) \right\},
\]
where \( \partial_x := \frac{\partial}{\partial x} = \partial_1 \) and \( \partial_y := \frac{\partial}{\partial y} = \partial_2 \) denote the weak partial derivatives with respect to \( x \) and \( y \), respectively.
\end{definition}

We define the degenerate elliptic operator
\[
\mathcal{L}u := \mathrm{div} \left( \frac{1}{y} \nabla u \right).
\]
A function \( u \) is called an \(\mathcal{L}\)-subsolution if it satisfies \( \mathcal{L}u \geq 0 \) in the distributional (or viscosity) sense. If both \( u \) and \( -u \) are \(\mathcal{L}\)-subsolutions, then \( u \) is referred to as an \(\mathcal{L}\)-solution.

We use \( \chi_A \) to denote the characteristic function of a set \( A \). 
For any real number \( a \in \mathbb{R} \), we write \( a^+ := \max(a, 0) \) and \( a^- := \min(a, 0) \). 
The symbol \( \nu \) will always refer to the inward-pointing unit normal vector of a surface. The notation \( \mathcal{H}^s \) denotes the \( s \)-dimensional Hausdorff measure.

We use the space of functions of bounded variation, denoted by \( BV(\Omega) \), consisting of functions \( f \in L^1(\Omega) \) whose distributional derivatives are vector-valued Radon measures. The total variation measure of such a function is denoted by \( |\nabla f| \). In particular, for a smooth open set \( U \subseteq \mathbb{R}^2 \), the measure \( |\nabla \chi_U| \) corresponds to the surface measure on \( \partial U \). We also use the concept of the reduced boundary, denoted \( \partial^\star U \), from geometric measure theory.

\subsection{Main Result}
We begin by introducing the class of non-minimizing solutions to the free boundary problem \eqref{OUR-problem}. 
In particular, we define the notion of a \textit{variational solution}, which includes certain critical points (such as saddle points) of the functional \eqref{singular-ACF} corresponding to the Euler-Lagrange equation \eqref{OUR-problem}. 
This formulation is inspired by the frameworks developed in \cite{kriventsov2023rectifiability, zbMATH06327034, MR1989835}. 
For simplicity, we assume $\lambda=1$ throughout the paper.

\begin{definition}[Variational solutions]
\label{Def-1-1}
Let \( \Omega \subseteq \mathbb{R}_+^{2} \) be an open set. 
A pair \( (\psi, I) \), with \( \psi: \Omega \to [0, +\infty) \) and \( \psi \in W^{1,2}_{\rW, \mathrm{loc}}(\Omega) \), and \( I: \Omega \to \{0,1\} \), is called a \textit{\textbf{variational solution}} (with constant \( C_V \)) if the following conditions are satisfied:
\begin{enumerate} 
\item $ \psi \in C^{0}(\Omega) \cap C^{2}(\Omega \cap \{\psi >0\}) $, with   $ \psi \geq 0 $ in  $ \Omega$, $\psi =0$ on $\{y=0\} \cap \Omega$, and
\begin{equation}\label{lip-condition}
    \dfrac1y|\nabla \psi | \leq C_{V}, \qquad \text{ in }\, \Omega;
\end{equation}
\item $I$ is a Borel measurable function; 
\item $\chi_{\{\psi >0\} } \leq I$, $ \mathcal{H}^2$-a.e in $\Omega$;
\item the first variation of the functional
\begin{equation*}
J(\tilde{\psi},\tilde{I}) = \int_{\Omega} \left( \left|\dfrac{\nabla \tilde{\psi}}{y} \right|^{2}+  \tilde{I} \right)y \, dX,
\end{equation*}
with respect to domain variations, vanishes at $ (\tilde{\psi},\tilde{I})=(\psi,I) $; that is, for every vector field $ \eta= (\eta_{1}, \eta_{2}) \in C_{c}^{1}(\Omega,\mathbb{R}^{2})$ such that $\eta_2=0$ on $\{y=0\}$, we have
\begin{align} \nonumber
0 & =  \dfrac{d}{dt}J \left(\psi(X+t\eta(X)),I(X+t\eta(X)) \right)_{\big
|_{t=0}}
\\ \label{inner-variation-1}
& = \int_{\Omega} \left(\frac{1}{y}\left|\nabla \psi  \right|^2 +  y I\right) \mathrm{div\,} \eta -\dfrac{2}{y}\nabla \psi \cdot D\eta \nabla \psi -\eta_2 \left(\dfrac{1}{y^2}|\nabla \psi|^2  -I \right)\, dX. 
\end{align}
\end{enumerate}
\end{definition}
\begin{remark}
The Lipschitz condition \eqref{lip-condition} follows from the boundedness of the velocity of the flow, which is known to hold at least for minimizing solutions of \eqref{singular-ACF}; see \cite{brock1993axiallyII}.
\end{remark}
Our main result establishes structural and regularity properties of variational solutions:
\begin{theorem}
\label{MAIN-THEOREM}
Let $(\psi, I)$ be a variational solution in a domain $\Omega$. Then:
\begin{itemize}
\item[(i)] Either $\psi \equiv 0$ or $I = \chi_{\{\psi > 0\}}$;

\item[(ii)] The whole topological free boundary $\partial\{\psi > 0\}$ is countably $\mathcal{H}^{1}$-rectifiable and has locally finite $\mathcal{H}^{1}$-measure.

\item[(iii)] The solution $\psi$ satisfies the following distributional equation
\[
\mathcal{L} \psi =  \mathcal{H}^{1} \llcorner \partial^\star \{\psi > 0\} + \frac{2}{\sqrt{\pi y}} \sqrt{\lim_{r \to 0^+} r^{-3} H_{X}(\psi,r)} \mathcal{H}^{1} \llcorner \Sigma_{H}.
\]
Here, $\partial^{\star}$ denotes the reduced boundary, $H_{X}(\psi,r)$ at the point $X=(x,y)$ is defined in \eqref{def-H-non-deg}, and $\Sigma_{H}$ denotes the set of all singular points $X_0 \in \partial \{\psi>0\}$ such that $ \psi(X) = a \,\big|\,(X-X_0)\cdot \nu\,\big| + \mathrm{o}(|X-X_0|)$, for some $a=a(X_0)>0$ and some unit vector $\nu=\nu(X_0)$.

\item[(iv)] For $\mathcal{H}^1$-a.e. singular point $X_0 \in \Sigma_{H}$, we have:
\[
\frac{r \int_{B_r^{+}(X_0)} \left( \frac{1}{y} \left|\nabla \psi(x,y)\right|^2 +  y \left(I - 1\right) \right) \, dx dy}{ \int_{\partial B_r^{+}(X_0)} \frac{1}{y} \psi^2(x,y) \, d\mathcal{H}^{1}} \to 1,
\]
and
\[
\frac{r \int_{B_r^{+}(X_0)} \frac{1}{y} \left|\nabla \left(\psi(x,y) - \alpha(X_0) \left|(x,y) - X_0 \right| \right) \cdot \nu(X_0)  \right|^2 \, dx dy}{\int_{\partial B_r^{+}(X_0)} \frac{1}{y} \psi^2(x,y) \, d\mathcal{H}^{1}} \to 0
\]
as $r \to 0$, for some unit vector $\nu(X_0)$. The constant $\alpha(X_0) = \sqrt{y_0/\pi}$ is an explicit normalizing factor. 
\end{itemize}
\end{theorem}

\section{ Preliminaries}

We begin with the proof of the first-variation formula \eqref{inner-variation-1}. 
Although a similar result can be found in \cite[Section 3.2]{zbMATH00837932} or \cite[Lemma 9.5]{MR4807210}, 
special care is required to handle the singularity at \( y = 0 \), which arises due to the axial symmetry of the problem. 
For the sake of completeness, we include the proof in the following lemma. 
The assumption \( \eta_2 = 0 \) on \( \{y = 0\} \), is crucial to ensure that the flat portion of the boundary of \( \Omega \) remains fixed under the variation.

\begin{lemma}
Suppose that $ \Omega \subset \mathbb{R}^{2}$ is a bounded open set and that  $\psi \in W^{1,2}_{\rW, \mathrm{loc}}(\Omega)$. 
Let $ \eta \in C_{c}^{1} (\Omega,\mathbb{R}^{2}) $ be a given vector field with compact support in $ \Omega $ satisfying $\eta_2=0$ on $\{y=0\}$.
Define the perturbation map $ p_{t} $ by 
\begin{equation*}
p_{t}(X)= X + t \eta (X), \qquad \text{for every $X \in \Omega$}.
\end{equation*}
Then the following hold:
\begin{enumerate}[label=(\roman*)]
\item for sufficiently small $t$  (depending on the vector field $\eta$), the map $ p_{t}: \Omega \to \Omega $ is a $C^1$-diffeomorphism and setting $ q_{t}:= p_{t}^{-1} $, the function $ \psi_{t} := \psi \circ q_{t} $ is well-defined and belongs to $ W^{1,2}_{\rW, \mathrm{loc}}(\Omega) $; 
\item the function $ t \to J(
\psi_t, I_t)$, where \( I_t := I \circ q_t \),  
is differentiable at $ t=0 $, and its derivative is given by:
\begin{equation*}
\begin{aligned}
\left.\frac{d}{dt} J(\psi_t, I_t)\right|_{t=0}
= & \int_{\Omega} \left(\frac{1}{y}\left|\nabla \psi  \right|^2 +  y I\right) \mathrm{div\,} \eta -\dfrac{2}{y}\nabla \psi \cdot D\eta \nabla \psi- \eta_2 \left(\dfrac{1}{y^2}|\nabla \psi|^2  -I \right)\, dX, 
\end{aligned}
\end{equation*}
which is referred to as the inner-variation of the functional $J(\tilde{\psi},\tilde{I})$
at $(\psi, I)$, in direction of the vector field $\eta \in C_{c}^{1} (\Omega,\mathbb{R}^{2})$.
\end{enumerate} 
\end{lemma}

\begin{proof} 
The first claim follows directly from the fact that \( \eta \in C_{c}^{1} (\Omega, \mathbb{R}^{2}) \). 
Therefore, we focus solely on proving the latter claim. 
By the change of variable $z=q_{t}(X)$, thus  $ X=p_{t}(z) = \left(p_1(z),p_2(z) \right)$, we get: 
\begin{align*}
\int_{\Omega} \dfrac{1}{y}|\nabla \psi_{t}|^{2}(X) + y\, I_t(X) \, dX &= \int_{\Omega} \left( \dfrac{1}{p_2(z)}\left| Dq_{t}(p_{t}(z))\nabla\psi(z)\right|^2+p_2(z)\, I(z)\,\right)  \left| \det Dp_{t}(z) \right|\,dz, \\
&=\int_{\Omega} \left(\dfrac{1}{p_2(z)}\nabla \psi(z)\cdot[Dp_{t}(z)]^{-T}[Dp_{t}(z)]^{-1}\nabla \psi (z) + p_2(z)\, I(z)\right) \, \left|\det Dp_{t}(z) \right|\,dz.
\end{align*}
We notice that 
\begin{align*}
& Dp_{t}=Id +tD\eta, \quad [Dp_{t}]^{-1}= Id - tD\eta +\mathrm{o}(t), \quad \det Dp_{t} = 1+ t \,\mathrm{div\,}\eta +\mathrm{o}(t), \\ 
& \dfrac{1}{p_2(z)}=\dfrac{1}{z_2}-t\frac{\eta_2(z)}{z_2^2}+\mathrm{o}(t),
\end{align*}
and we continue calculation as follows (for $t$ small):
\begin{align*}
&=\int_{\Omega} \left( \left(\dfrac{1}{z_2}-t\dfrac{\eta_2(z)}{z_2^2} \right)\nabla \psi(z)\cdot [Id - tD\eta]^{T} [Id - tD\eta]\nabla \psi (z) +\left(z_2+ t\eta_2(z)\right)\, I(z) \right) \, \left|1+ t \,\mathrm{div\,} \eta \right|\,dz + \mathrm{o}(t), \\
&=\int_{\Omega} \left( \left(\dfrac{1}{z_2}-t\dfrac{\eta_2(z)}{z_2^2}\right)\left(|\nabla \psi(z)|^2-2t\nabla \psi(z) \cdot  D\eta \nabla \psi(z) \right)+\left(z_2+t\eta_2(z)\right)
I(z)\right) \, \left(1+ t \, \mathrm{div\,} \eta \right) \,dz + \mathrm{o}(t), \\
&=\int_{\Omega} \dfrac{1}{z_2}|\nabla \psi(z)|^2-\dfrac{2t}{z_2}\nabla \psi(z) \cdot D\eta \nabla \psi(z) 
-t\dfrac{\eta_2(z)}{z_2^2}|\nabla \psi(z)|^2+z_2I(z)+t\eta_2(z)I(z)
\\
& \qquad \qquad  +t \, \mathrm{div\,} \eta \left(\dfrac{1}{z_2} |\nabla \psi(z)|^2  +z_2  I(z) \right) \,dz   +\mathrm{o}(t).
\end{align*}
Thus
\begin{align*}
\int_{\Omega} \dfrac{1}{y}|\nabla \psi_{t}|^{2}(X)&+y\, I_t(X) \, dX =  \int_{\Omega} \dfrac{1}{y}|\nabla \psi|^{2}+ y I(X) \, dX  \\
& +t \int_{\Omega}\left( \dfrac{1}{y}|\nabla \psi|^{2} +  y I\right) \mathrm{div \,} \eta -\dfrac{2}{y}\nabla \psi \cdot D\eta \nabla \psi-\dfrac{\eta_2}{y^2}|\nabla \psi|^2 + \eta_2  I\, dX + \mathrm{o}(t), 
\end{align*}
which concludes the proof of \textit{(ii)}.
\end{proof}

\medskip

\begin{proposition} 
\label{32}
Let $ (\psi,I) $ be a  variational solution in $ \Omega $. 
Then $ \psi $ is a $\mathcal{L}$-subsolution in $ \Omega $, and $\mathcal{L}$-solution in $ \{\psi > 0\} $.
\end{proposition}
\begin{proof}
 Take $\eta \in C_c^{\infty}
(\{\psi > 0\}, \bR^2)$. Then $I \geq \chi_{\{\psi>0\}} = 1 $ Lebesgue-a.e. on the support of $\eta$, so 
\begin{equation}
\label{eq-eq-Eqq0}
\int_{ \{\psi>0\}} \left(\frac{1}{y}\left|\nabla \psi  \right|^2 +  y \right) \mathrm{div\,} \eta -\dfrac{2}{y}\nabla \psi \cdot D\eta \nabla \psi -\dfrac{1}{y^2} |\nabla \psi|^2 \eta_2 + \eta_2   \, dX=0.
\end{equation}
Note that $\eta=0$ on $y=0$ since $\psi$ vanishes there. Therefore, integrating by parts in the first term of  \eqref{eq-eq-Eqq0} yields: 
\begin{align*}
-\int_{\{\psi>0\}} \dfrac{2\eta_1}{y} \left( \psi_{x} \psi_{xx} + \psi_y \psi_{xy} \right)+  \int_{\{\psi>0\}} \dfrac{\eta_2}{y^2} \left( \psi_{x}^2 + \psi_y^2 \right) 
 -\int_{\{\psi>0\}} \dfrac{2 \eta_2}{y}\left(\psi_{x} \psi_{xy}+\psi_y \psi_{yy}\right) - \int_{\{\psi>0\}}  \eta_2; 
\end{align*}
and by integrating by parts on the second term of \eqref{eq-eq-Eqq0}, we'll have:
\begin{align*}
&\int_{\{\psi>0\}} \dfrac{4}{y}\eta_1\psi_{x} \psi_{xx} + \int_{\{\psi>0\}}\eta_1 \left (\dfrac{-2}{y^2}\psi_{x}\psi_y +\dfrac{2}{y}\psi_{xy} \psi_y + \dfrac{2}{y}\psi_{x} \psi_{yy} \right)+ \int_{\{\psi>0\}}\dfrac{2}{y}\eta_2 \left(\psi_{xx}\psi_y + \psi_{x}\psi_{yx} \right)\\ & \qquad + \int_{\{\psi>0\}}\eta_2 \left(\dfrac{-2}{y^2}\psi_y^2 +\dfrac{4}{y}\psi_y \psi_{yy} \right).
\end{align*}
Now by adding all of them up including the third term of \eqref{eq-eq-Eqq0}, we will have: 
\begin{align*}
2\int_{\{\psi>0\}}\left( \eta \cdot \nabla \psi \right) \,\mathrm{div\,}\left(\dfrac{\nabla \psi}{y} \right) = 2 \int_{\{\psi>0\}} \left(\eta \cdot \nabla \psi\right) \, \mathcal{L}\psi  = 0. 
\end{align*}
This implies that $\psi$ is an $\mathcal{L}$-solution in $\{\psi >0\} \cap \{ \nabla \psi \neq 0 \}$.  
Hence, $\mathcal{L}\psi=0$ in $\{\psi >0\}$ due to the assumption $\psi \in C^2( \{\psi >0\})$.
Moreover, as \( \psi \) is continuous in \( \Omega \), it follows that \( \psi \) is an \( \mathcal{L} \)-subsolution in \( \Omega \), for instance by verifying that it satisfies the subsolution property in the viscosity sense.
\end{proof}

The following equality for variational solutions will be useful in the subsequent sections.

\begin{lemma}  
\label{useful-lemma-01}  
Let \( (\psi, I) \) be a variational solution in \( \Omega \). 
Then, for every ball \( \overline{B_s^+(X_0)} \subset \Omega \), the following equality holds:  
\begin{equation}  
\label{r1}  
\int_{B_s^+(X_0)} \frac{1}{y} |\nabla \psi|^2 \, dX = \int_{\partial B_s^+(X_0)} \frac{1}{y} \psi \nabla \psi \cdot \nu \, d\mathcal{H}^{1}, 
\end{equation}  
where \( \nu \) denotes the outward unit normal on \( \partial B_s^+(X_0) \).  
\end{lemma} 

\begin{proof}
Using the dominated convergence theorem, Proposition \ref{32}, and integration by parts applied to the sublevel sets \( \{\psi > \epsilon\} \), for a sequence of \( \epsilon >0 \) such that \( \partial \{\psi > \epsilon\} \) is smooth (guaranteed for Lebesgue-a.e. \( \epsilon \) by Sard's theorem), we obtain:
\begin{align*}
\nonumber
\int_{B_s^+(X_0)} \dfrac{1}{y} |\nabla \psi|^2 \, dX &= \lim_{\epsilon \to 0^+ }\int_{B_s^+(X_0)} \dfrac{1}{y} \nabla \psi \cdot \nabla (\psi - \epsilon)^+ \, dX \\ 
\nonumber
&= - \lim_{\epsilon \to 0^+} \int_{B_s^+(X_0)} (\psi-\epsilon)^+ \, \mathrm{div}\left(\dfrac{\nabla \psi}{y} \right) \, dX + \lim_{\epsilon \to 0^+} \int_{\partial B_s^+(X_0)} \dfrac{1}{y} (\psi - \epsilon)^+ \nabla \psi \cdot \nu \, d\mathcal{H}^{1} \\ 
\nonumber
&= \int_{\partial B_s^+(X_0)} \dfrac{1}{y} \psi \nabla \psi \cdot \nu \, d\mathcal{H}^{1}.
\end{align*}
\end{proof}

\begin{lemma}
\label{BV-estimate}
Let $ (\psi,I) $ be a variational solution on $ \Omega $, and $ B_{2r}^+(X_0) \Subset \Omega $. Then $ y I \in BV(B_{r}^+(X_0)) $, and 
\begin{equation*}
\int_{B_{r}^+(X_0)} \left|\nabla \left(y I \right) \right|\leq C r(y_0+r),
\end{equation*}
where the constant $C$ depends only on $C_{V}$.
\end{lemma}

\begin{proof}
We derive the following basic estimate for \(\mathcal{L}\psi\), understood as a non-negative Borel measure. 
Let \(\zeta\) be a cutoff function such that \(\zeta \equiv 1\) on \(B_{3r/2}^+(X_0)\) and \(\zeta \equiv 0\) outside \(B_{2r}^+(X_0)\). Then, we have  
\[
\mathcal{L}\psi\left(B_{3r/2}^+(X_0)\right) \leq \int_{\Omega} \zeta \, d\mathcal{L}\psi = -\int_{\Omega} \frac{1}{y} \nabla \psi\, \cdot\, \nabla \zeta \leq \int_{B_{2r}^+ \setminus B_{3r/2}^+} \left| \frac{1}{y} \nabla \psi \right| \, | \nabla \zeta | \leq  C_VCr.
\]
Let \(\phi_t\) be a standard mollifier, and define \(\psi_t = \phi_t * \psi\). Then \(\psi_t\) is smooth, and from \eqref{inner-variation-1}, we obtain:  
\[
\begin{aligned}
& \lim_{t \to 0^+} \int_{\Omega} \frac{1}{y}\left|\nabla \psi_t \right|^2 \operatorname{div} \eta 
- \frac{2}{y} \nabla \psi_t \cdot D\eta \nabla \psi_t 
- \eta_2 \left(\frac{1}{y^2}|\nabla \psi_t|^2 - I \right) \, dX \\
& \quad = \int_{\Omega} \frac{1}{y}\left|\nabla \psi \right|^2 \operatorname{div} \eta 
- \frac{2}{y} \nabla \psi \cdot D\eta \nabla \psi 
- \eta_2 \left(\frac{1}{y^2}|\nabla \psi|^2 -  I \right) \, dX \\
& \quad = - \int_{\Omega} y I \, \operatorname{div} \eta \, dX,
\end{aligned}
\]
for every \(\eta \in C_c^{0,1}(B_{3r/2}^+(X_0), \mathbb{R}^2)\). 
Using integration by parts (as in the proof of Proposition \ref{32}), we can rewrite the expression as follows:  
\[
\int_{\Omega} \frac{1}{y}\left|\nabla \psi_t \right|^2 \operatorname{div} \eta - \frac{2}{y} \nabla \psi_t \cdot D\eta \nabla \psi_t - \eta_2 \left(\frac{1}{y^2}|\nabla \psi_t|^2 -  I \right) \, dX = 2 \int_{\Omega} \eta \cdot \nabla \psi_t \, \mathcal{L}\psi_t + \int_{\Omega}  \eta_2 I.
\]
Invoking the Lipschitz condition \eqref{lip-condition} and Proposition \ref{32} which guarantees $\mathcal{L}\psi \ge 0$, we have:  
\[
\left| \int_{\Omega} \frac{1}{y}\left|\nabla \psi_t \right|^2 \operatorname{div} \eta 
- \frac{2}{y} \nabla \psi_t \cdot D\eta \nabla \psi_t 
- \eta_2 \left(\frac{1}{y^2}|\nabla \psi_t|^2 -  I \right) \, dX \right| \leq CC_V  \sup |\eta| \, r (y_0+r) + C \sup |\eta| r^2.
\]
In particular,  
\[
\left|  \int_{\Omega} y I \, \operatorname{div} \eta \, dX \right| \leq C \sup |\eta|  \, r(y_0+r),
\]
which implies that \(yI \in BV(B_r^+(X_0))\), along with the stated estimate.  
\end{proof}

\section{Monotonicity Formulas and Classification of Blowups}
\label{s4}

Given a variational solution $(\psi,I)$ in $\Omega$, with free boundary $\Gamma := \Omega \cap \partial\{\psi > 0\}$, 
we refer to the points of the free boundary that lie on the axis $y=0$ as \textit{degenerate} points. 
The set of all such points is denoted by
\begin{equation*}
S_{\psi} := \Gamma \cap \{y=0\}.
\end{equation*}

It is important to note that the gradient condition $|\nabla \psi| =  y$ on the free boundary governs the behavior of the solution near degenerate points. 
More specifically, near a degenerate point $X_0 = (x_0, 0)$, the stream function $\psi$ typically exhibits quadratic growth, i.e., \(\psi(X) \sim r^2\) as \(X \to X_0\). 
This behavior is fundamentally different from that near other free boundary points.

To study the free boundary in regions away from the symmetric axis, we introduce the set of \textit{non-degenerate} free boundary points.  
Formally, this set is given by
\begin{equation*}
N_{\psi} := \left\{X_0 = (x_0, y_0) \in \Gamma :\,  y_0 > 0 \right\}.
\end{equation*}
We also define the \textit{growth exponent} \(\mu\), which captures the asymptotic growth rate of \(\psi\) near free boundary points, as follows:
\begin{equation}
\label{mu-constant}
\mu = \mu(X_0) := 
\begin{cases}
2, & \text{if } y_0 = 0, \\
1, & \text{if } y_0 > 0.
\end{cases}
\end{equation}
Now, for $X_0 \in \Omega$, we define the following energy and boundary functionals:
\begin{equation*}
D_{X_0}(\psi,r)= \int_{B_{r}^{+}(X_0)} \frac{1}{y}|\nabla \psi|^{2}+ y I \, dX,
\end{equation*}
and 
\begin{equation}
\label{def-H-non-deg}
H_{X_0}(\psi,r)=\int_{\partial B_{r}^{+}(X_0)} \frac{1}{y}\psi^{2}  \,d\mathcal{H}^{1},
\end{equation}
where $B_{r}^{+}(X_0)=B_r(X_0)\cap \{y > 0\} \subseteq \Omega$.  
We now introduce the following monotonicity-based functional:

\medskip

- For degenerate points $X_0$ on the axis  $\{y=0\}$,
\begin{equation*}
M_{X_0}(\psi, r) := r^{-3} D_{X_0}(\psi, r) - 2r^{-4} H_{X_0}(\psi, r).
\end{equation*}

- For non-degenerate points $X_0 \in \{y>0\}$,
\begin{equation*}
M_{X_0 }(\psi, r) := r^{-2} D_{X_0}(\psi, r) - r^{-3} H_{X_0}(\psi, r).
\end{equation*}

\medskip

To unify the notation, we use the growth constant $\mu = \mu(X_0)$, and rewrite both expressions in the unified form:
\[
M_{X_0}(\psi, r) := r^{-\mu - 1} D_{X_0}(\psi, r) - \mu r^{-\mu - 2} H_{X_0}(\psi, r).
\]
When the context is clear, we may omit the argument $ \psi $ and write simply $M_{X_0}(r)$. 
Moreover, when the center point $X_0$ is fixed or understood, we further abbreviate to $M(r)$. The same conventions apply to $D$ and $H$.

The primary use of the monotonicity formula is to identify explicit blow-up limits. 
We demonstrate this in Proposition \ref{10.1} and Proposition \ref{10.4}.
In particular, we apply the formula to show that any limit of a blow-up subsequence is a homogeneous function of degree $\mu$.

\begin{lemma}[Monotonicity Formula]
\label{M-F-01}
Let $(\psi, I)$ be a variational solution in $B_r^+(X_0) \subseteq \Omega$ with \(\psi(X_0)=0\). Then the function $M_{X_0}(r)$ is absolutely continuous and  non-decreasing when $y_0=0$ and almost non-decreasing when $y_0>0$. 
Moreover, for Lebesgue-a.e. $r \in (0, r_0)$, it satisfies the identity:
\begin{align}
\label{m-f-non-degenerate}
\frac{dM_{X_0}(r)}{dr} = & 2r^{-\mu-1} \int_{\partial B_r^+(X_0)} \frac{1}{y} \left( \nabla \psi \cdot \nu - \frac{\mu}{r} \psi \right)^2 \, d\mathcal{H}^{1} \nonumber\\
&+(2-\mu) r^{-\mu-2}\int_{B_{r}^{+}(X_0)}  (y-y_0) I-\dfrac{y-y_0}{y^2}|\nabla \psi |^2  dX 
+ (2-\mu) r^{-\mu-3} \int_{\partial B_r^+(X_0)}\dfrac{y-y_0}{y^2}\psi^2(X) \, d\mathcal{H}^1.
\end{align}
\end{lemma}

\begin{proof}
Consider the function \( \xi_{\epsilon}(t) \) for sufficiently small \( \epsilon > 0 \) as follows:
\begin{equation*}
\xi_{\epsilon}(t):= \max\left(0,\min\left(1,\dfrac{t}{\epsilon}\right)\right)=
\begin{cases}
1 \quad & t \ge \epsilon, \\
 t/\epsilon  \quad & 0 < t <\epsilon, \\
0 \quad & t \le 0.
\end{cases}
\end{equation*}
We now define the vector field:
$$\eta(X):=\xi_{\epsilon}(r-|X-X_0|)\left[\begin{matrix}x-x_0 \\[8pt]
    \xi_{\epsilon}(y)(y-y_0)
\end{matrix}\right],$$ 
which satisfies the boundary condition $\eta_2=0$ on $y=0$, and use it as a test function in the domain variation formula \eqref{inner-variation-1}. 
Passing to the limit as $\epsilon \to 0$, and using assumption \eqref{lip-condition}, we obtain:
\begin{equation}
\label{mono-eq1}
\int_{B_r^+(X_0)} 2 y  I  +  (y-y_0) I-\dfrac{y-y_0}{y^2}|\nabla \psi |^2  dX -r\int_{\partial B_r^+(X_0)}\left[ \left(\frac{1}{y}\left|\nabla \psi  \right|^2 +  y I\right)-\frac{2}{y}  \left(\nabla \psi \cdot \nu \right)^2 \right] d \mathcal{H}^1=0.
\end{equation}

\medskip
We now compute the derivative of the monotonicity functional. 
Using the change of variables $Z=\dfrac{X-X_0}{r}$, a direct computation yields 
\begin{align}
\nonumber
\dfrac{d \left(r^{-\mu-2} H_{X_0}(r) \right)}{dr}
& =\dfrac{d}{dr}\left( r^{-\mu-1}\int_{\partial B_{1}^{+}}\dfrac{1}{y_0+rz_2}\psi^2(X_0+rZ) \, d\mathcal{H}^1\right) \\ \nonumber
&=-(\mu+1)r^{-\mu-2}\int_{\partial B_{1}^{+}}\dfrac{1}{y_0+rz_2}\psi^2(X_0+rZ) \, d\mathcal{H}^1 - r^{-\mu-1}\int_{\partial B_{1}^{+}} \dfrac{z_2}{(y_0+rz_2)^2}\psi^2(X_0+rZ) \, d\mathcal{H}^1 \\ \nonumber
&\qquad \qquad +r^{-\mu-1}\int_{\partial B_{1}^{+}} \dfrac{2\psi}{y_0+rz_2} \nabla \psi(X_0+rZ)\cdot Z \, d\mathcal{H}^1\\ \label{dH_0}
&=-(\mu+1)r^{-\mu-3}\int_{\partial B_{r}^{+}(X_0)}\dfrac{1}{y}\psi^2(X) \, d\mathcal{H}^1 - r^{-\mu-3}\int_{\partial B_{r}^{+}(X_0)}\dfrac{y-y_0}{y^2}\psi^2(X) \, d\mathcal{H}^1 \\ \nonumber
& \qquad \qquad +2r^{-\mu-2}\int_{\partial B_{r}^{+}(X_0)}\dfrac{1}{y}\psi \nabla \psi(X)\cdot \nu\, d\mathcal{H}^1.
\end{align}
Note that when \( X_0 \in N_\psi \) and \( r > y_0 \), the domain \( \partial B_1^+ \) corresponds to the truncated unit circle \( z_2 > -y_0/r \). However, since the integrand vanishes on \( y = 0 \), the result remains unaffected.
Also, we have
\begin{equation}
\label{e1}
\dfrac{d \left( r^{-\mu-1} D_{X_0}(r) \right) }{dr}=-(\mu+1)r^{-\mu-2}\int_{B_{r}^{+}(X_0)} \left(\frac{1}{y}|\nabla \psi|^{2}+ y I\right)\, dX +r^{-\mu-1}\int_{\partial B_{r}^{+}(X_0)} \left(\frac{1}{y}|\nabla \psi|^{2}+ yI\right)\, d\mathcal{H}^1.
\end{equation}

When $X_0\in S_\psi$ is a degenerate point, we have $\mu = 2$ and $y_0 = 0$. Invoking equations \eqref{mono-eq1} and \eqref{e1}, we obtain:
\begin{equation}
\label{Eq6}
\dfrac{d \left( r^{-3} D_{X_0}(r) \right)}{dr}=-4r^{-4}\int_{B_{r}^{+}(X_0)} \frac{1}{y}|\nabla \psi|^{2} \, dX + r^{-3}\int_{\partial B_r^{+}(X_0)} \frac{2}{y} \left(\nabla \psi \cdot \nu \right)^2 \, d\mathcal{H}^1.
\end{equation}
Now, applying equation \eqref{dH_0} together with \eqref{r1}, we get:
\begin{align*}
\dfrac{dM_{X_0}(r)}{dr} & = \dfrac{d \left( r^{-3} D_{X_0}(r) \right)}{dr}-2\dfrac{d \left( r^{-4} H_{X_0}(r)\right)}{dr}= 2r^{-3} \int_{\partial B_r^+(X_0)} \frac{1}{y} \left( \nabla \psi \cdot \nu-\dfrac{2}{r}\psi\right)^2 \,d\mathcal{H}^{1}.
\end{align*}

When $X_0 \in N_\psi$ is a non-degenerate point, we have $\mu = 1$. Once again, applying equations \eqref{mono-eq1} and \eqref{e1}, together with  \eqref{r1} yields:
\begin{align}
\dfrac{d \left( r^{-2} D_{X_0}(r) \right)}{dr}  =& -2r^{-3}\int_{\partial B_{r}^{+}(X_0)}\dfrac{1}{y}\psi \nabla \psi(X)\cdot \nu\, d\mathcal{H}^1
+ r^{-3}\int_{B_{r}^{+}(X_0)}  (y-y_0) I-\dfrac{y-y_0}{y^2}|\nabla \psi |^2  dX\notag \\ \label{derv-D-N}
& +r^{-2}\int_{\partial B_r^+(X_0)} \frac{2}{y}  \left(\nabla \psi \cdot \nu \right)^2  d \mathcal{H}^1.
\end{align}
Combining this with \eqref{dH_0}, we obtain
\begin{align*}
\dfrac{dM_{X_0}(r)}{dr}  = & 2r^{-2} \int_{\partial B_r^+(X_0)} \frac{1}{y} \left( \nabla \psi \cdot \nu-\dfrac{\psi}{r}\right)^2 \,d\mathcal{H}^{1} + r^{-3}\int_{B_{r}^{+}(X_0)}  (y-y_0) I-\dfrac{y-y_0}{y^2}|\nabla \psi |^2  dX \nonumber \\
&+ r^{-4} \int_{\partial B_r^+(X_0)}\dfrac{y-y_0}{y^2}\psi^2(X) \, d\mathcal{H}^1.
\end{align*}
\end{proof}

\begin{remark}\label{rmrk-almost-mono}
From relation \eqref{m-f-non-degenerate}, when \( X_0 \in S_\psi \) is a degenerate point, we have \( \mu(X_0) = 2 \), and \( M_{X_0}(r) \) is monotone. However, when \( X_0 \in N_\psi \) is a non-degenerate point, \( \mu(X_0) = 1 \), and there exists a universal constant \( C_M = C( C_V) \) such that the function \( M_{X_0}(r) + C_M r \) is monotone. This behavior is referred to as \textit{almost monotonicity}.
\end{remark}

\begin{definition} \label{d41}
Let \( r_n > 0 \) be a sequence converging to \( 0 \) as \( n \to +\infty \). 
The corresponding blow-up sequence \( \{\psi_n\} \) is defined for \( |X| < \frac{1}{r_n} \dist(X_0, \partial \Omega) \) as follows:   
\[
\psi_n(X) := \dfrac{\psi(X_0 + r_n X)}{r_n^{\mu(X_0)}}, \qquad \text{for any } X_0 \in \Gamma.
\] 
Here, $\mu(X_0)$ is the growth constant defined in \eqref{mu-constant}. 
If the functions \( \psi_n \) are uniformly bounded in \( W^{1,2}_{\rW, \mathrm{loc}}(B^+_r) \), then there exists a subsequence converging weakly in $W^{1,2}_{\rW, \mathrm{loc}}(\mathbb{R}^2_+)$ to a limit function $\psi_0$, which we refer to as a blow-up limit of $\psi$ at $X_0$. If $X_0 \in N_\psi$, the domain extends to all of $\mathbb{R}^2$, and convergence is in $W^{1,2}_{\rW, \mathrm{loc}}(\mathbb{R}^2)$.
\end{definition}

We use the monotonicity formula  to demonstrate that blow-up limits of variational solutions are one-homogeneous functions. 
Our goal is to prove that for every $X_0 \in N_{\psi}$ and for a variational solution $(\psi,I)$  in $\Omega$, the blow-up limit of $\psi$  can be classified into three types, as described in the following proposition.

\begin{proposition}[Blow-up limits of the variational solutions at non-degenerate points]
\label{10.1}
Let $X_0 \in N_{\psi}$ and let $(\psi,I)$ be a variational solution  on
$B_{1}^+(X_0)$, then the following statements hold:

\begin{enumerate}
\item
The rescalings $\psi_n$ converge to $\psi_0$ strongly in $W^{1,2}_{\rW,\mathrm{loc}}(\mathbb{R}^2)$.

\item
Every blow-up limit $\psi_0$ of $\psi_n$ is a one-homogeneous function, i.e. $\psi_0(rX) = r\psi_0(X)$
for all $X \in \mathbb{R}^2$ and $r > 0$. 

\item
If $ \psi_0 \not\equiv 0$, then the only blow-up limits are:
$$ \psi_0(X)= y_0 (X\cdot\nu)_+ \qquad \text{and} \qquad \psi_0(X)=  A |X\cdot\nu|$$ 
for some unit vector $\nu$ and some positive constant $A>0$.
\end{enumerate}
\end{proposition}

\begin{proof}
\textit{(1)} Considering the rescaled functions \( \psi_n(X) = \frac{\psi(X_0+r_nX)}{r_n} \) and \( I_n(X) = I(X_0+r_nX) \); these functions also satisfy the variational solution properties in the ball \( B_R \) when \( n \) is sufficiently large (however, a rescaled version of \eqref{inner-variation-1} must be used, where the weight $1/y$ replaced by $1/(y_0+r_ny)$). 
In particular,  \( \psi_n \) remains uniformly bounded in  \( W^{1,2}_{\rW} \) with respect to \( n \).
For any sequence \( r_n \to 0 \), we can extract a subsequence (still denoted \( r_n \)) along which \( \psi_{n} \) converges locally uniformly and weakly in \( W^{1,2}_{\rW, \mathrm{loc}} \) to a Lipschitz function \( \psi_0: \mathbb{R}^2 \to [0,+\infty) \).  
Furthermore, this convergence can be strengthened on the set  \( A = \{\psi_0 > 0\} \). 
Indeed, for each \( Y \in A \), there exists a small ball \( B_{\epsilon}(Y) \subseteq A \) such that \( \psi_{n} > 0 \) in \( B_{\epsilon}(Y) \) for all sufficiently large \( n \) and 
\[
\Delta \psi_n = \dfrac{r_n}{y_0+r_n y} \partial_y \psi_n, \quad \text{ in } B_{\epsilon}(Y).
\] 
By elliptic regularity, we conclude that \( \nabla \psi_{n} \to \nabla \psi_0 \) in \( B_{\epsilon}(Y) \),  
which implies that \( \psi_0 \) is a harmonic function in $A$.

Now, we prove the strong convergence in \( W^{1,2}_{\rW,\mathrm{loc}}(\mathbb{R}^2) \). Take any \( \eta \in C_c^{\infty}(\mathbb{R}^2) \) with \( \eta \geq 0 \), and a sequence \( \epsilon \to 0^+ \) such that \( \{ \psi_0 > \epsilon \} \) is smooth on a neighborhood of \( \supp \eta \). 
We have
\[
\begin{aligned}
\int_{\supp \eta}  |\nabla \psi_0|^2 \eta \, dX & = \lim_{\epsilon \to 0^+} \int_{\supp \eta} \nabla \psi_0 \cdot \nabla \left(\psi_0 - \epsilon \right)_+ \eta \, dX \\
& = - \lim_{\epsilon \to 0^+}  \int_{\supp \eta}  \mathrm{div} \left(\eta \nabla \psi_0 \right) \left(\psi_0 - \epsilon \right)_+ \, dX \\
& = - \lim_{\epsilon \to 0^+}  \int_{\supp \eta} \eta \Delta \psi_0 \left(\psi_0 - \epsilon \right)_+ \, dX - \lim_{\epsilon \to 0^+}  \int_{\supp \eta}  \nabla \eta \cdot \nabla \psi_0  \left(\psi_0 - \epsilon \right)_+ \, dX \\
& = - \int_{\supp \eta}  \psi_0 \nabla \eta \cdot \nabla \psi_0 \, dX,
\end{aligned}
\]
where the first and last steps used dominated convergence and the middle step that \( \psi_0 \) is harmonic on \( \{ \psi_0 > \epsilon \} \subseteq A \). The same computation is valid with \( \psi_{n} \) in place of \( \psi_0 \) with weight $1/(y_0+r_ny)$, and so we see that
\[
\begin{aligned}
\int_{\supp \eta}  |\nabla \psi_0|^2 \eta \, dX &= - \int_{\supp \eta} \psi_0 \nabla \eta \cdot \nabla \psi_0 \, dX \\
& = -\lim_{r_n \to 0} \int_{\supp \eta}  \frac{y_0}{y_0+r_ny} \psi_{n} \nabla \eta \cdot \nabla \psi_{n} \, dX = \lim_{r_n \to 0} \int_{\supp \eta} \frac{y_0}{y_0+r_ny} |\nabla \psi_{n}|^2 \eta \, dX.
\end{aligned}
\]
This implies that \( \psi_{n} \to \psi_0 \) strongly in \( W^{1,2}_{\rW,\mathrm{loc}}(\mathbb{R}^2) \) as \( n \to +\infty \).

\textit{(2)} By integrating the equality \eqref{m-f-non-degenerate} with $\mu = 1$ over the interval $(r_n \rho_1, \, r_n \rho_2)$, for $0 < \rho_1 < \rho_2$, we obtain:
\begin{align*}
\int_{r_n\rho_1}^{r_n\rho_2}\dfrac{dM_{X_0}(r)}{dr} \, dr   = & \int_{r_n\rho_1}^{r_n\rho_2} 2r^{-2} \int_{\partial B_r^+(X_0)} \frac{1}{y}  \left( \nabla \psi \cdot \nu-\dfrac{\psi}{|X-X_0|}\right)^2 \,d\mathcal{H}^{1}dr \\
&  + \int_{r_n\rho_1}^{r_n\rho_2} r^{-3}\left(\int_{B_{r}^{+}(X_0)}  (y-y_0) I 
- \dfrac{y-y_0}{y^2}|\nabla \psi |^2 \right)\; dX dr \\
&  + \int_{r_n\rho_1}^{r_n\rho_2} r^{-4} \int_{\partial B_r^+(X_0)}\dfrac{y-y_0}{y^2}\psi^2(X) \, d\mathcal{H}^1 dr\\
 = & \int_{B_{r_n\rho_2}(X_0)\setminus B_{r_n\rho_1}(X_0)}\dfrac{2}{y}|X-X_0|^{-2} \left( \nabla \psi \cdot \nu-\dfrac{\psi}{|X-X_0|}\right)^2 \; dX\\
&  + \int_{r_n\rho_1}^{r_n\rho_2} r^{-3}\left(\int_{B_{r}(X_0)} (y-y_0) I 
- \dfrac{y-y_0}{y^2}|\nabla \psi |^2 \right)\; dX dr \\
&  + \int_{B_{r_n\rho_2}(X_0)\setminus B_{r_n\rho_1}(X_0)} |X-X_0|^{-4} \dfrac{y-y_0}{y^2}\psi^2(X) \, dX.
\end{align*}
After rescaling, we can compute
\begin{align*}
\int_{r_n\rho_1}^{r_n\rho_2} \dfrac{dM_{X_0}(r)}{dr}\; dr  = & \int_{B_{\rho_2}\setminus B_{\rho_1}}\dfrac{2}{y_0 +r_n \overline{y}} \left| \overline{X} \right|^{-4} \left( \nabla \psi_n \cdot \overline{X}-\psi_n\right)^2 \;d\overline{X} \\
&  +  r_n \int_{\rho_1}^{\rho_2} r^{-3}\left(\int_{B_{r}}  \overline{y}  I_n
- \dfrac{ \overline{y}}{\left(y_0 +r_n \overline{y}\right)^2} \left|\nabla \psi_n  \right|^2 \right) \, d\overline{X} dr \\
&  + \int_{B_{\rho_2}\setminus B_{\rho_1}} \left| \overline{X} \right|^{-4} \dfrac{r_n \overline{y}}{\left(y_0 +r_n \overline{y}\right)^2}\psi_n^2   \, d\overline{X}.
\end{align*}
so
$$
\limsup_{r_n\to0}\int_{B_{\rho_2}\setminus B_{\rho_1}}\dfrac{2 \left|\overline{X}\right|^{-4}}{y_0 +r_n \overline{y}} \left( \nabla \psi_n \cdot \overline{X}-\psi_n\right)^2 \;d\overline{X} \leq \lim_{r_n\to0} M_{X_0}(r_n\rho_2)-M_{X_0}(r_n\rho_1) = 0.
$$
The desired homogeneity then follows from the strong convergence of $\psi_n$ to $\psi_0$.

\textit{(3)}
We now derive the form of the function $\psi_0(X)$ in polar coordinates, under the assumption that $\{\psi_0 > 0\}$ is nonempty. 
From \textit{(1)}, we know that $\psi_0$ is harmonic in $\{\psi_0>0\}$.
This together with the homogeneity, i.e. $\psi_0(X) = r \phi(\theta)$, gives $\phi''+\phi =0$ in $\{\phi>0\}$, where $\phi$ is a $2\pi$-periodic function. 
 
Suppose $\phi>0$ on an interval $(0,a)$ with $\phi(0)=\phi(a)=0$, then $a=\pi$ and $\phi(\theta)=A_+\sin(\theta)$ on $[0,\pi]$ for some constant $A_+>0$. 
If $\phi$ is not identically zero on $[\pi, 2\pi]$, a similar argument yields $\phi(\theta) = A_- \sin(\theta)$ on $[\pi,2\pi]$ for some $A_->0$. 
Therefore, in general, 
$$ \psi_0(X)=A_+ \left( X \cdot \nu \right)_++A_- \left(X \cdot \nu \right)_-,$$ 
for some unit vector $\nu$, with at least one of the constants $A_{\pm} $ nonzero.

On the other hand, the rescaled functions \( \psi_n \) satisfy the following domain variation formula:
\begin{align*}
\int \left(\frac{1}{y_0+r_n y} \left|\nabla \psi_n \right|^2 + (y_0+r_n y) I_n \right)\mathrm{div\,} \xi & - \frac{2}{y_0+r_n y}\nabla \psi_n \cdot D \xi \nabla \psi_n \\
& - r_n \xi_2\left( \frac{1}{(y_0+r_n y)^2} \left|\nabla \psi_0 \right|^2  -  I_n  \right) \, dX = 0.
\end{align*}
Considering the strong $L^1_{\mathrm{loc}}$-limit of $I_{n}$ along a subsequence, and denoting this limit by $I_0 \in \{0, 1\}$ (Lebesgue-a.e. in $\mathbb{R}^2$) and using the fact that $I_0 \ge \chi_{\{\psi_0>0\}}$ together with the local uniform convergence of $\psi_n$ to $\psi_0$, we get $I_0 = 1$ Lebesgue-a.e. in $\{\psi_0 > 0\}$. 
Now, after passing to the limit (notice that we have the strong convergence of $\psi_n$'s and also the compact embedding of $BV$ in $L^1$), we obtain:
\begin{equation}
\label{Thm53-Eq-A}
\int \left( \frac{1}{y_0} \left|\nabla \psi_0 \right|^2 +  y_0 I_0 \right) \mathrm{div\,} \xi  - \frac{2}{y_0}\nabla \psi_0 \cdot D \xi \nabla \psi_0 \, dX =0.
\end{equation}
We observe from \eqref{Thm53-Eq-A} that if $\psi_0 = 0$ on $\supp\xi$, then one has
$$
\int I_0\ \mathrm{div\,} \xi \, dX =0;
$$
which gives that $I_0$ is a constant in each connected component of $\{\psi_0 = 0\}$.

First, consider the case when one of the constants $A_\pm$ is zero; for instance, $\psi_0(X)=A_+ \left( X \cdot \nu \right)_+$.
We already know that $I_0=1$ on $\{X\cdot \nu>0\}$ and that $I_0$ is constant on $\{X \cdot \nu < 0\}$, denoted as $\overline{I}_0$.
The outer normal to $\partial\{\psi_0 > 0\} =\{X\cdot\nu=0\}$ is the constant normal $-\nu$. 
Let $\xi(X) := \zeta (X) \nu$ for any $\zeta \in C^{1}_c(\bR^2)$ and insert $\xi$ into the equation \eqref{Thm53-Eq-A}. Integrating by parts gives that:
\begin{equation*}
\begin{aligned}
0 = &\int_{\{X\cdot\nu>0\}}   \left(\frac{A_+^2}{y_0}  + y_0 \right) \mathrm{div\,} \xi - \frac{2A_+^2}{y_0}\nu \cdot D \xi \nu \, dX + \int_{\{X\cdot\nu < 0\}}  y_0 \overline{I}_0\ \mathrm{div\,} \xi  \, dX  \\
 =  & \int_{\{X\cdot\nu=0\}} \left(  \frac{A_+^2}{y_0} +  y_0 (\overline{I}_0 -1)\right) \zeta \, d \mathcal{H}^1.
\end{aligned}
\end{equation*}
Since this holds for arbitrary $\zeta$, it follows that $A_+^2=  y_0^2(1-\overline{I}_0)$ on $\{X\cdot\nu=0\}$. Because we assumed $A_+>0$, we deduce that $\overline{I}_0=0$, and hence $\psi_0(X)= y_0 \left(X \cdot \nu \right)_+$. 

Now, suppose that both of $A_{\pm}$ are not zero. We will show that $A_+=A_- $. In this case, obviously we have $I_0=1$ Lebesgue-a.e. in $\mathbb{R}^2$. 
Therefore, using \eqref{Thm53-Eq-A} and testing with a vector field  $\xi(X) := \zeta (X) \nu$  implies that 
\[
0= \int \left( \mathrm{div\,} \xi  - 2 \nu \cdot D \xi \nu \right)\left(A_+^2\chi_{\{X\cdot\nu>0\}} + A_-^2\chi_{\{X\cdot\nu < 0\}}\right) \, dX =\int_{\{X\cdot\nu=0\}} \left( A_+^2-A_-^2 \right) \zeta \, d \mathcal{H}^1.
\]
Since this holds for arbitrary $\zeta$, we conclude $A_+ = A_-:=A$. Therefore, $\psi_0(X)=A |X \cdot \nu|$. 
\end{proof}

In the next proposition, we once again apply the monotonicity formula to show that blow-up limits at degenerate points are two-homogeneous functions, and we explicitly determine the form of the blow-up limit.

\begin{proposition}[Blow-up limits of the variational solutions at degenerate points]
\label{10.4}
Let $X_0 \in S_{\psi}$ and let $(\psi,I)$ be a variational solution  on
$B_{1}^+(X_0)$, then the following statements hold:

\begin{enumerate}
\item
$\psi_n$ converges to $\psi_0$ strongly in $W^{1,2}_{\rW,\mathrm{loc}}(\mathbb{R}^2_+)$.

\item
Every blow-up limit $\psi_0$ of $\psi_n$ is a two-homogeneous function, i.e. $\psi_0(rX) = r^2\psi_0(X)$
for each $X \in \mathbb{R}_+^2$ and $r > 0$. 

\item
If $ \psi_0 \not\equiv 0$, then we have $\psi_0(X)=A y^2$ for some $A>0$.
\end{enumerate}
\end{proposition}

\begin{proof}
The proofs of \textit{(1)} and \textit{(2)} are analogous to those in the non-degenerate case.
For \textit{(3)}, we refer to Proposition 4.3 in \cite{du2023free1}.
\end{proof}

In the following proposition, we characterize all possible values of $M_{X_0}(0^+)$.

\begin{proposition}
\label{Highest-frequency-V}
Let \( (\psi, I) \) be a variational solution in \( \Omega \), and suppose \( B_1^{+}(X_0) \subseteq \Omega \). Then \( M_{X_0}(\psi, 0^+) = -\infty \) if and only if \( \psi(X_0) > 0 \). If \( \psi(X_0) = 0 \), then
\begin{equation}  
\label{H-F-Interval}
M_{X_0}(\psi,0^+) = \lim_{r \to 0} \frac{\left| \left\{X \in B_r^{+}(X_0) \,:\, I(X) = 1 \right\} \right|_{\rW}}{r^{\mu+1}} \in \left[0,  \omega(X_0) \right],  
\end{equation}  
where $\mu=\mu(X_0)$ is the growth exponent defined in \eqref{mu-constant}, and \( \left|\,\cdot\,\right|_{\rW} \)  denotes the weighted volume with respect to the measure \( y \, dX \), given by 
\[
\left| A \right|_{\rW}:= \int_{A} y \, dX.
\]  
The energy upper bound \( \omega(X_0) \) is given by
\[
\omega(X_0):=  r^{-\mu-1}\int_{B_r^{+}(X_0)} y \, dX = 
\begin{cases}
\dfrac{2}{3}, \quad & X_0 \in S_{\psi}, \\
y_0\pi, \quad & X_0 \in N_{\psi} \text{ and } r < y_0.
\end{cases}
\]  
Furthermore, the map \( X_0 \mapsto M_{X_0}(\psi, 0^+) \) is upper semi-continuous in \( \Omega \).
\end{proposition}

\begin{proof}
The fact that \( M_{X_0}(r) \to -\infty \) when \( \psi(X_0) > 0 \) is straightforward. 
In this case, \( X_0 \) must lie away from the axis \( y = 0 \), and thus \( \mu(X_0) = 1 \). 
Moreover, the quantity \( r^{-2} D_{X_0}(r) \) remains bounded, while
\( r^{-1} H_{X_0}(r) \to \frac{1}{y_0} \psi^2(X_0) |\partial B_1^+| \).
Therefore, we may assume \( \psi(X_0) = 0 \) and proceed to establish \eqref{H-F-Interval}.

First, in this case, \( M_{X_0}(r) \) is bounded in terms of the constant \( C_V \). More precisely, when $X_0\in S_\psi$ we have  
\[
D_{X_0}(r) \leq \int_{B_r^+(X_0)} C_V^2 y + yI \, dX \le (1+C_V^2) \omega(X_0) r^3,
\]  
while  
\[
H_{X_0}(r) \leq \int_{\partial B_r^+(X_0)}C_V^2 y^3 \, d\mathcal{H}^1 \le C_V^2 \left|\partial B_1^+ \right|r^4.
\]  
In particular, \( M_{X_0}(0^+) \) is finite.
When $X_0\in N_\psi$ and $r<y_0/2$, we similarly obtain $D_{X_0}(r) \leq (1+C_V^2)\omega(X_0)r^2$ and $H_{X_0}(r) \leq 5y_0C_V^2\left|\partial B_1^+ \right| r^3$.

Consider the rescaled functions \( \psi_r(X) = \frac{\psi(X_0+rX)}{r^\mu} \) and \( I_r(X) = I(X_0+rX) \).
These functions also satisfy the variational solution properties in half-ball $B_R^+$ when $X_0\in S_\psi$ and in the ball \( B_R \) when \( X_0 \in N_\psi \) and \( r \) is sufficiently small. In the latter case, a rescaled version of \eqref{inner-variation-1} must be used, where the weight $1/y$ is replaced by $1/(y_0+ry)$. 

By Proposition \ref{10.1} and Proposition \ref{10.4}, for any sequence \( r_k \to 0 \) we can extract a subsequence (still denoted \( r_k \)) along which \( \psi_{r_k} \) converges strongly in  \( W^{1,2}_{\rW} \) to a function \( \psi_0 \), which is homogeneous of degree \( \mu(X_0) \), an \( \mathcal{L} \)-solution in \( \{\psi_0 > 0\} \) when \( X_0 \in S_\psi \), and harmonic when \( X_0 \in N_\psi \).

From the strong convergence of \( \psi_{r_k} \) in \( W^{1,2}_{\rW}(B_1) \), we  obtain in case $X_0\in N_\psi$
\[
 \int_{B_1} \frac{1}{y_0} |\nabla \psi_0|^2 \, dX -   \int_{\partial B_1} \frac{1}{y_0} \psi_0^2 \, d\mathcal{H}^{1} = \lim_{r_k \to 0}  \int_{B_1} \frac{1}{y_0+r_ky} |\nabla \psi_{r_k}|^2 \, dX -    \int_{\partial B_1} \frac{1}{y_0+r_ky} \psi_{r_k}^2 \, d\mathcal{H}^{1}.
\]
Arguing as in \eqref{r1}, the left-hand side vanishes, since \( \psi_0 \) is a homogeneous harmonic function on \( \{\psi_0 > 0\} \).  
It follows that
\[
0 = \lim_{r_k \to 0} r_k^{-2} \int_{B_{r_k}(X_0)} \frac{1}{y} |\nabla \psi|^2 \, dX -  r_k^{-3} \int_{\partial B_{r_k}(X_0)} \frac{1}{y} \psi^2 \, d\mathcal{H}^{1} = \lim_{r_k \to 0} M_{X_0}(\psi, r_k) - \frac{\left| \{X \in B_{r_k}(X_0) \, : \, I(X) = 1\} \right|_{\rW}}{r_{k}^{2}}.
\]
Therefore,
\[
M_{X_0}(\psi, 0^+) = \lim_{r \to 0} \frac{\left| \left\{X \in B_r(X_0) \, : \, I(X) = 1 \right\} \right|_{\rW}}{r^{2}}.
\]
Since this holds along every sequence \( r_k \to 0 \), the conclusion follows. 
The argument for $X_0 \in S_\psi$ proceeds in an analogous manner.

The upper semi-continuity of the map  \( X_0 \mapsto M_{X_0}(\psi,0^+) \) follows directly from the fact that it is the monotone limit of continuous functions. 
Indeed, when $y_0>0$, let \( \delta > 0 \) and \( K < +\infty \). Then, Lemma \ref{M-F-01} implies that
\[
M_{X}(\psi,0^+) \leq M_{X}(\psi,r) + C_Mr \leq M_{X_0}(\psi,r) + C_Mr + \frac{\delta}{2} \leq 
\begin{cases}
M_{X_0}(\psi,0^+) + \delta, & \quad \text{if \,} M_{X_0}(\psi,0^+) > -\infty, \\
-K, & \quad \text{if \,} M_{X_0}(\psi,0^+) = -\infty,
\end{cases}
\]
provided we first fix \( r > 0 \) (depending on \( X_0 \)), and then take \( |X - X_0| \) sufficiently small.

However, when $y_0=0$, since $M_X(\psi, r)$ is not continuous in $X$ near $X_0$; thus, we must distinguish between sequences approaching $X_0$ along $\{y=0\}$ and along $\{y>0\}$.
If $\{y=0\} \ni X_n\to X_0$, same reasoning as above shows upper semi-continuity.  

If $\{y>0\} \ni X_n\to X_0$, and $\psi(X_n)>0$ then $M_{X_n}(\psi, 0^+)=-\infty$ which implies that
$\limsup_{n \to + \infty} M_{X_n}(\psi, 0^+) \le M_{X_0}(\psi, 0^+)$.
In the case $\psi(X_n)=0$,  we have already shown that $M_{X_n}(\psi, 0^+) \in [0, \omega(X_n)]$, so
\[
\limsup_{n \to + \infty} M_{X_n}(\psi, 0^+) \le \limsup_{n \to + \infty} \omega(X_n) = 0 \le M_{X_0}(\psi, 0^+).
\]
\end{proof}

\section{The Frequency Formula}

Given a variational solution \( (\psi, I) \) in \( \Omega \), the \textbf{frequency formula} is defined as  
\[
\cN_{X_0}(\psi, r) = \frac{r D_{X_0}(\psi, r) - r\int_{B_r^{+}(X_0)} y \, dX}{ H_{X_0}(\psi, r)},
\]
and any \( B_r^{+}(X_0) \subseteq \Omega \), where the denominator is positive. 

For convenience, we also introduce the \textbf{volume difference} as 
\[
\cV_{X_0}(\psi, r) = \frac{r\int_{B_r^{+}(X_0)}  y \left(1 - I \right) \, dX}{ H_{X_0}(\psi, r)} \geq 0.  
\] 
With these definitions, we have the identity  
\begin{equation}
\label{nv}  
\cN_{X_0}(\psi, r) + \cV_{X_0}(\psi, r) = \frac{r \int_{B_r^{+}(X_0)} \frac{1}{y} |\nabla \psi|^2 \, dX}{ H_{X_0}(\psi, r)},
\end{equation}  
which corresponds to the classical Almgren frequency formula for harmonic functions. 
When there is no ambiguity, we suppress the parameters \( \psi \) and \( X_0 \) and simply write \( \mathcal{N}(r) \) and \( \mathcal{V}(r) \).

We will establish the monotonicity of the frequency formula in separate subsections.

\subsection{Monotonicity of the Frequency Formula: Degenerate Case}

\begin{lemma}
\label{nprim}  
Let \( (\psi, I) \) be a variational solution in \( \Omega \), and suppose \( B_{r_0}^+(X_0) \subseteq \Omega \). Assume that \( X_0 \in S_{\psi} \) and that \( M_{X_0}(\psi, t) \geq  \omega(X_0) \) for some \( t < r_0 \). 
Then, on the interval \( (t, r_0) \), the frequency formula \( \cN_{X_0}(\psi, r) \) satisfies \( \cN_{X_0}(\psi, r) \geq 2 \), and \( \cN_{X_0}(\psi, r) \) is absolutely continuous and non-decreasing. Moreover, the derivative of the frequency formula is given by  
\begin{align}  
\frac{d}{dr} \cN_{X_0}(\psi, r) &= \frac{2r^{-1}}{ H_{X_0}(\psi, r)} \int_{\partial B_{r}^{+}(X_0)} \frac{1}{y} \left| \nabla \psi \cdot \nu - \frac{\psi}{r}\left(\cN_{X_0}(\psi,r) + \cV_{X_0}(\psi,r) \right) \right|^2 \, d\mathcal{H}^{1} 
\nonumber \\  
&\quad + \frac{2}{r} \left(  
\cV_{X_0}^2(\psi, r) + \cV_{X_0}(\psi, r) \left( \cN_{X_0}(\psi, r) -2 \right) \right)
\label{npr}  
\end{align}
for Lebesgue-a.e. \(r \in (t, r_0) \).
\end{lemma}

\begin{proof}
Without loss of generality, we may assume \( X_0 = \mathbf{0} \). 
First, notice that by Lemma \ref{M-F-01}, we have \( M_{\mathbf{0}}(r) \geq M_{\mathbf{0}}(t) \geq  \omega(\mathbf{0}) \), for $r \in (t,r_0)$ which implies \( \cN_{\mathbf{0}}(r) \geq 2 \). Furthermore, this lemma implies that \( \cN_{\mathbf{0}}(r) \) is absolutely continuous. 
Recalling the computations in \eqref{Eq6} and \eqref{dH_0}, we obtain
\[
(r^{-3}D)' = -4r^{-4}\int_{B_{r}^{+}} \frac{1}{y}|\nabla \psi|^{2} \, dX + r^{-3}\int_{\partial B_r^{+}} \frac{2}{y} \left(\nabla \psi \cdot \nu \right)^2 \, d\mathcal{H}^1,
\]
and
\[
(r^{-4}H)' = -4r^{-5}\int_{\partial B_{r}^{+}}\dfrac{1}{y}\psi^2 \, d\mathcal{H}^1 +2r^{-4}\int_{\partial B_{r}^{+}}\dfrac{1}{y}\psi \nabla \psi\cdot \nu\, d\mathcal{H}^1.
\]
Writing \(\mathcal{N}=\mathcal{N}_{\mathbf{0}}\) for short, we have
\[
\cN'  = \frac{d}{dr} \left( \frac{r^{-3}D -  \omega}{r^{-4}H} \right) = \frac{(r^{-3}D)'}{r^{-4}H} - \frac{(r^{-4}H)'\left(r^{-3}D -  \omega \right)}{(r^{-4}H)^2}
= \frac{(r^{-3}D)'-(r^{-4}H)'\cN}{r^{-4}H}.
\]

Expanding the numerator and simplifying yields
\begin{align} \nonumber
\cN' & = \frac{1}{H} \left(-4\int_{B_{r}^{+}} \frac{1}{y}|\nabla \psi|^{2} \, dX + r\int_{\partial B_r^{+}} \frac{2}{y} \left(\nabla \psi \cdot \nu \right)^2 \, d\mathcal{H}^1 \right) \\ \nonumber
& \quad - \frac{\cN}{H} \left( -4r^{-1}\int_{\partial B_{r}^{+}}\dfrac{1}{y}\psi^2 \, d\mathcal{H}^1 +2\int_{\partial B_{r}^{+}}\dfrac{1}{y}\psi \nabla \psi\cdot \nu\, d\mathcal{H}^1 \right)\\ \label{np}
& = -\frac{4}{r}(\cN+\cV) + \frac{2r}{H}\int_{\partial B_r^{+}} \frac{1}{y} \left(\nabla \psi \cdot \nu \right)^2 \, d\mathcal{H}^1  + \frac{4}{r}\cN - \frac{2}{r}\cN(\cN+\cV),
\end{align}
where in the last line we have used \eqref{r1} and \eqref{nv}, together with the definition of \(H\).
The boundary integral can be rewritten as
\begin{align*}
 \frac{2r}{H}\int_{\partial B_r^{+}} \frac{1}{y} \left(\nabla \psi \cdot \nu \right)^2 \, d\mathcal{H}^1 =
 \frac{2r}{H}\int_{\partial B_r^{+}} \frac{1}{y} \left(\nabla \psi \cdot \nu -\frac{\psi}{r}(\cN+\cV)\right)^2 \, d\mathcal{H}^1 + \frac{2}{r}(\cN+\cV)^2.
\end{align*}
Substituting this into \eqref{np} gives
\[
\cN' = \frac{2r}{H}\int_{\partial B_r^{+}} \frac{1}{y} \left(\nabla \psi \cdot \nu -\frac{\psi}{r}(\cN+\cV)\right)^2 \, d\mathcal{H}^1  + \frac{2}{r}\cV(\cN+\cV-2).
\]

Finally, to confirm that \( \cN \) is nondecreasing, note that \( \cV \geq 0 \) and, as established earlier, \( \cN \geq 2 \). These conditions guarantee that all terms in the expression for \( \cN' \) are non-negative.
\end{proof}

\subsection{Almost Monotonicity of the Frequency Formula: Nondegenerate Case}
Before proving the monotonicity of the frequency formula, we need a modified monotonicity result for $M$, which provides a lower estimate for $M(r)$.

\begin{lemma}\label{auxil-lem-almgrn}
Let $(\psi, I)$ be a variational solution on $\Omega$, and suppose $B_{r_0}^+(X_0) \subseteq \Omega$. 
Fix a constant $\delta_0>0$, and assume that $X_0=(x_0, y_0) \in N_\psi$ with $y_0 \ge \delta_0$. 
Moreover, suppose that \( M_{X_0}(\psi, 0^+) =  \omega(X_0) \). 
Then there exists a constant $C_\star=C(\delta_0)$ such that the inequality
\[
M_{X_0}(r) + C_\star r^{-5/2} H_{X_0}(r) \ge \omega(X_0),
\]
holds for $r<\min(r_0, y_0/128)$. 
\end{lemma}
\begin{proof}
Recall \eqref{m-f-non-degenerate} and \eqref{dH_0}, and note that $\int_{B_r(X_0)} (y-y_0)\,dX = 0 \quad \text{for } r<y_0$, thus for $r<y_0$ we can compute
\begin{align*}
\frac{d}{dr}(M+C_\star r^{-5/2}H) \ge &\  r^{-3}\int_{B_{r}(X_0)}  (y-y_0) (I-1)-\dfrac{y-y_0}{y^2}|\nabla \psi |^2  dX + r^{-4}\int_{\partial B_{r}(X_0)}\dfrac{y-y_0}{y^2}\psi^2(X) \, d\mathcal{H}^1  \\
 & -\frac{3C_\star}{2}r^{-7/2} H + 2C_\star r^{-5/2} \int_{B_r(X_0)} \dfrac{1}{y}|\nabla \psi |^2  dX - C_\star r^{-7/2}\int_{\partial B_{r}(X_0)}\dfrac{y-y_0}{y^2}\psi^2(X) \, d\mathcal{H}^1 \\
 \ge &-\frac{r^{-2}}{y_0-r} \int_{B_{r}(X_0)}  y(1-I)+\dfrac{1}{y}|\nabla \psi |^2  dX 
  - \frac{r^{-3}+C_\star r^{-5/2}}{y_0-r}H  -\frac{3C_\star}{2}r^{-7/2} H  \\
& + 2C_\star r^{-5/2} \int_{B_r(X_0)} \dfrac{1}{y}|\nabla \psi |^2  dX,
\end{align*}
where we have also used the inequality $ -ry/(y_0-r) \le y-y_0 \le ry/(y_0-r)$ when $r<y_0$.
Now, using the definition of $M$, we obtain
\begin{align*}
\frac{d}{dr}\left(M+C_\star r^{-5/2}H \right) \ge & \left(2C_\star r^{-1/2}-\frac{1}{y_0-r}\right)  \left(M-\omega(X_0)+C_\star r^{-5/2}H \right) + \left(2C_\star r^{-5/2}-\frac{2r^{-2}}{y_0-r} \right)\int_{B_{r}(X_0)}  y(1-I)\,dX \\
&+ \left(\frac{C_\star}{2}r^{-7/2} - 2C_\star^2 r^{-3} - \frac{2r^{-3}}{y_0-r}\right)H.
\end{align*}
Choosing $C_\star= \sqrt{2/\delta_0}$ and restricting to $r \le y_0/128$, the coefficients of all three terms are positive. Hence,
\[
\frac{d}{dr} \left(M-\omega(X_0)+C_\star r^{-5/2}H \right) \ge \left(2C_\star r^{-1/2}-\frac{1}{y_0-r} \right)  \left(M-\omega(X_0)+C_\star r^{-5/2}H \right). 
\]
This proves the statement, since $r^{-3}H(r)$ is bounded and therefore 
\[\lim_{r \to 0} M(r) -\omega(X_0)  + C_\star r^{-5/2} H(r) = 0.
\]
\end{proof}

\begin{lemma}[non-degenerate case] 
\label{nprim-N}  
Let \( (\psi, I) \) be a variational solution on \( \Omega \), and suppose \( B_{r_0}^+(X_0) \subseteq \Omega \). 
Fix a constant $\delta_0>0$, and assume that \( X_0=(x_0, y_0) \in N_\psi \) with $y_0\ge \delta_0$.
Moreover, suppose that  \( M_{X_0}(\psi, 0^+) =  \omega(X_0) \). 
Then, the frequency function \( \cN_{X_0}(\psi, r) \) satisfies 
\[
\cN_{X_0}(\psi, r) \;\geq\; 1 - C_\star r^{1/2}.
\]
Furthermore, the function \( \cN_{X_0}(\psi, r) \) is absolutely continuous and almost monotone. 
In particular, there exists a constant $C=C(\delta_0)$ such that for Lebesgue-a.e. $r < \min(r_0, y_0/128, \delta_0/2)$,
\begin{equation}\label{ndg-almgren}
(y_0-r)^2\dfrac{d}{dr}\left[(y_0-r)^{-2}\cN\right]
\ge \dfrac{2r}{ H} \int_{\partial B_{r}(X_0)} \dfrac{1}{y}\left(\nabla \psi \cdot \nu-\dfrac{\psi }{r } \left(\cN+\cV \right)\right)^2 \, d\mathcal{H}^1 + \frac{2}{r}\left(\cV^2 + \cV \left(\cN- 1- \frac{r}{y_0-r} \right)\right).
\end{equation} 
\end{lemma}
\begin{proof}
From Lemma \ref{auxil-lem-almgrn},  we deduce the first statement of the lemma. Indeed,
\begin{equation}\label{nprim-N-eq1}
\cN(r) = \frac{rD - r^3\omega(X_0)}{H} = \frac{r^3M + H -r^3\omega(X_0)
}{H} \ge 1 - C_\star r^{1/2}.
\end{equation}

To analyze the derivative of the frequency function, recall from \eqref{Eq6} and \eqref{dH_0} that
\begin{align*}
(r^{-2}D)' = & -2r^{-3}\int_{B_{r}(X_0)} \frac{1}{y}|\nabla \psi|^{2} \, dX 
+ r^{-3}\int_{B_{r}(X_0)}  (y-y_0) I-\dfrac{y-y_0}{y^2}|\nabla \psi |^2  dX\nonumber \\
& +r^{-2}\int_{\partial B_r(X_0)} \frac{2}{y}  \left(\nabla \psi \cdot \nu \right)^2  d \mathcal{H}^1.
\end{align*}
and
\[
(r^{-3}H)' =  -2r^{-4}\int_{\partial B_{r}(X_0)}\dfrac{1}{y}\psi^2(X) \, d\mathcal{H}^1 
-r^{-4}\int_{\partial B_{r}(X_0)}\dfrac{y-y_0}{y^2}\psi^2(X) \, d\mathcal{H}^1
+2r^{-3}\int_{\partial B_{r}(X_0)}\dfrac{1}{y}\psi \nabla \psi(X)\cdot \nu\, d\mathcal{H}^1;
\]

Thus, computing $\cN'$ as in the degenerate case gives
\begin{align*}
\cN' = &  \frac{d}{dr} \left( \frac{r^{-2}D -  \omega(X_0)}{r^{-3}H} \right) 
= \frac{(r^{-2}D)'-(r^{-3}H)'\cN }{r^{-3}H}\\
= &  \dfrac{2r}{H} \int_{\partial B_{r}(X_0)} \dfrac{1}{y}\left(\nabla \psi \cdot \nu-\dfrac{\psi}{r} \left(\cN+\cV \right)\right)^2 \, d\mathcal{H}^1 + \frac{2}{r}\left(\cV^2 + \cV(\cN-1)\right) \\
& + \frac{1}{H}\int_{B_{r}(X_0)}  (y-y_0) (I-1)-\dfrac{y-y_0}{y^2}|\nabla \psi |^2  dX + \frac{r^{-1}\cN}{H}\int_{\partial B_{r}(X_0)}\dfrac{y-y_0}{y^2}\psi^2(X) \, d\mathcal{H}^1,
\end{align*}
where we have used the relation $\int_{B_r(X_0)}(y-y_0)\,dX=0$ for $r<y_0$ in the last line. 
Applying the inequality $ -ry/(y_0-r) \le y-y_0 \le ry/(y_0-r)$ for $r<y_0$, and using that $\cN \ge 0$ when $r< \delta_0/2$ (by \eqref{nprim-N-eq1}), we obtain
\begin{align*}
\cN' \ge & \dfrac{2r}{H} \int_{\partial B_{r}(X_0)} \dfrac{1}{y}\left(\nabla \psi \cdot \nu-\dfrac{\psi}{r} \left(\cN+\cV \right)\right)^2 \, d\mathcal{H}^1 + \frac{2}{r}\left(\cV^2 + \cV(\cN-1)\right) 
\\& - \frac{r}{(y_0-r)H} \int_{B_{r}(X_0)}  y(1-I)+\frac{1}{y}|\nabla \psi |^2  dX - \frac{r^{-1}\cN}{H}\int_{\partial B_{r}(X_0)}\dfrac{r}{y(y_0-r)}\psi^2(X) \, d\mathcal{H}^1\\
= &\dfrac{2r}{H} \int_{\partial B_{r}(X_0)} \dfrac{1}{y}\left(\nabla \psi \cdot \boldsymbol{\nu}-\dfrac{\psi}{r} \left(\cN+\cV \right)\right)^2 \, d\mathcal{H}^1 + \frac{2}{r}\left(\cV^2 + \cV(\cN-1)\right) - \frac{2(\cN+\cV)}{y_0-r}
\end{align*}
Therefore,
\[
\cN' + \frac{2\cN}{y_0-r} \ge\dfrac{2r}{H} \int_{\partial B_{r}(X_0)} \dfrac{1}{y}\left(\nabla \psi \cdot \nu-\dfrac{\psi}{r} \left(\cN+\cV \right)\right)^2 \, d\mathcal{H}^1 + \frac{2}{r}\left(\cV^2 + \cV \left(\cN-1 - \frac{r}{y_0-r} \right)\right).
\]
\end{proof}

\begin{remark}\label{rmrk-almost-almgren}
Under the assumptions of Lemma~\ref{nprim-N}, we get
\begin{align*}
(y_0-r)^{-2}\cN(r) - (y_0-s)^{-2} \cN(s) \ge  & \int_s^r \frac{2\cV^2}{t(y_0-t)^{2} } \ dt - C_\star \left( \int_s^r  \frac{2\cV^2}{t(y_0-t)^{2} }  \ dt\right)^{1/2}
\left( \int_s^r  \frac{2}{(y_0-t)^{2} }  \ dt\right)^{1/2} \\
& - \left( \int_s^r  \frac{2\cV^2}{t(y_0-t)^{2} }  \ dt\right)^{1/2} \left( \int_s^r  \frac{2t}{(y_0-t)^{4} }  \ dt\right)^{1/2} \\
 \ge & \int_s^r \frac{2\cV^2}{t(y_0-t)^{2} } \ dt  - 2 \left( C(\delta_0) (r-s) \int_s^r  \frac{2\cV^2}{t(y_0-t)^{2} }  \ dt\right)^{1/2} \\
 \ge & - C(\delta_0) (r-s)
\end{align*}
Thus the almost monotonicity relation
\begin{equation}\label{almost-almgren}
\cN(s) \le \cN(r) + C_N (r-s) (1+\cN(r))
\end{equation}
holds for $ s\le r < \min(y_0/128, \delta_0/2)$.
\end{remark}
\begin{remark}
\label{r2}
By multiplying through and applying \eqref{r1} and \eqref{nv}, the right hand side of formulas \eqref{npr} and \eqref{ndg-almgren} can be rewritten as 
\begin{align*}
\frac{2r}{H} \int_{\partial B_{r}^{+}(X_0)} \frac{1}{y} \left| \nabla \psi \cdot \nu - \frac{\psi}{r} \cN \right|^2 \, d\mathcal{H}^{1}+ \frac{2}{r}  \cV \left( \cN- \mu(X_0) - \frac{r\left(2-\mu(X_0)\right)}{y_0-r} \right).
\end{align*}
This alternative formulation is more convenient when handling the first term on the right-hand side. Conversely, \eqref{npr} is more beneficial when focusing on the \( \cV_{X_0}^2(\psi, s) \) term or when finer control over the coefficient associated with the homogeneity order of \( \psi \) is needed.
In subsequent sections, we will repeatedly integrate the above identity, for which the following estimate, when $X_0 \in N_\psi$, will be useful.
Using the estimate $\cN(t) \ge 1- C_\star t^{1/2}$, we obtain
\[
\begin{aligned}
\int_s^r \frac{2t}{H} \int_{\partial B_{t}^{+}(X_0)} \frac{1}{y} \left| \nabla \psi \cdot \nu - \frac{\psi}{r} \cN \right|^2 \, d\mathcal{H}^{1}dt & \le \left(\frac{y_0-s}{y_0-r}\right)^{2}\cN(r) - \cN(s)  + (y_0-s)^2 \int_s^r \frac{2C_\star  \cV}{t^{1/2}(y_0-t)^2} + \frac{2\cV}{(y_0-t)^3}\, dt\\
& \le \cN(r)-\cN(s) + C(r-s) \cN(r) +  C r^{1/2} \left(\int_s^r \frac{2\cV^2}{t(y_0-t)^2}\, dt \right)^{1/2}\\
& \le \cN(r)-\cN(s) + Cr \cN(r) + Cr^{1/2} \left(\cN(r)^{1/2} + r^{1/2} \right) \\
& \le \cN(r)-\cN(s) + Cr^{1/2}\cN(r) + Cr,
\end{aligned} 
\]
where we have used  the computation in Remark \ref{rmrk-almost-almgren} to estimate the integral term involving $\cV$.
Here the constant $C$ depends on $\delta_0$, and the estimate holds for sufficiently small $r$, as in Lemma \ref{nprim-N}.
\end{remark}

\subsection{Further Properties of the Frequency Formula}
\begin{lemma}[degenrate case]
\label{42}
Let $(\psi, I)$ be a variational solution on $\Omega$, and $B_r^{+}(X_0) \subseteq \Omega$. Assume that $ X_0 \in S_\psi$,
$M_{X_0}(\psi, r/2) \geq  \omega(X_0)$ and that $ \cN_{X_0}(\psi, r) \leq N_+ $ for some $N_+ > 2$. Then, there is a constant
$C=C(N_+)$ such that for any $s\in[r/2, r]$, we have:
\begin{align*}
s^{-4}H_{X_0}(\psi, s) \leq r^{-4} H_{X_0}(\psi, r) \leq Cs^{-4}H_{X_0}(\psi, s).
\end{align*}
\end{lemma}
\begin{proof}
We can assume that \( X_0 = \mathbf{0} \).  From the assumptions \( \cN_0(s) \geq 2 \), and since \( \cV_{\mathbf{0}}(s) \geq 0 \), we have \( \cN_{\mathbf{0}}(s) + \cV_{\mathbf{0}}(s) \geq 2 \). Additionally, we have the following expression for the derivative:
\begin{align*}
\dfrac{d}{ds} s^{-4}H_{\mathbf{0}}(\psi, s) & =-4s^{-5}\int_{\partial B_{s}^{+}}\dfrac{1}{y}\psi^2(X) \, d\mathcal{H}^1 +2s^{-4} \int_{\partial B_{s}^{+}} \dfrac{1}{y}\psi \nabla \psi(X)\cdot \boldsymbol{\nu}\, d\mathcal{H}^1;
\end{align*}
which, with the help of \eqref{r1} and \eqref{nv}, gives:
\[
\frac{d}{ds} s^{-4} H_{\mathbf{0}}(\psi, s) = \frac{H_{\mathbf{0}}}{s^5} \left( -4 + 2(\cN_{\mathbf{0}} + \cV_{\mathbf{0}}) \right) \geq 0.
\]
This yields the lower bound of this lemma. For the upper bound, we can proceed directly using Lemma \ref{nprim} as follows:
\[
\int_{r/2}^{r} \frac{2\cV_{\mathbf{0}}^2(s)}{s} \, ds \leq \int_{r/2}^{r} \frac{d}{ds} \cN_{\mathbf{0}}(s) \, ds = \cN_{\mathbf{0}}(r) - \cN_{\mathbf{0}} \left( r/2 \right) \leq N_+ - 2;
\]
hence,
\[
\int_{r/2}^{r} \cV_{\mathbf{0}}(s) \, ds \leq \sqrt{r/2 \int_{r/2}^{r} \cV_{\mathbf{0}}^2(s) \, ds} \leq \sqrt{ r^2\int_{r/2}^{r} \frac{\cV_{\mathbf{0}}^2(s)}{2s} \, ds} \leq \frac{r\sqrt{N_+ - 2}}{2}.
\]
Now, noting that \( s \geq r/2 \), by integrating the following formula
\[
\frac{d}{ds} \log s^{-4} H_{\mathbf{0}}(s) = \frac{\frac{d}{ds} s^{-4}H_{\mathbf{0}}(s)}{s^{-4}H_{\mathbf{0}}(s)} = \frac{2}{s} \left( \cN_{\mathbf{0}} + \cV_{\mathbf{0}} - 2 \right) \leq \frac{4}{r} \left( N_+ + \cV_{\mathbf{0}} - 2 \right)
\]
over \([s,r]\), we obtain:
\[
\log r^{-4}H_{\mathbf{0}}(r) - \log s^{-4}H_{\mathbf{0}}(s) \leq \frac{4}{r} \int_s^r \left( N_+ + \cV_{\mathbf{0}} - 2 \right) \, ds \leq 2 \sqrt{N_+ - 2} + 2 \left( N_+ - 2 \right).
\]
Thus, we have:
\[
r^{-4}H_{\mathbf{0}}(r) \leq s^{-4} H_{\mathbf{0}}(s) e^{2 \sqrt{N_+ - 2} + 2 (N_+ - 2)}.
\]
This completes the proof.
\end{proof}

\begin{lemma}[non-degenrate case]
\label{42-N}
Let $(\psi, I)$ be a variational solution on $\Omega$, and suppose $B_r(X_0) \subseteq \Omega$. 
Assume that $ X_0 \in N_\psi$ with $y_0 \ge \delta_0$. 
Moreover, suppose that $M_{X_0}(\psi, 0+) =  \omega(X_0)$ and that $ \cN_{X_0}(\psi, r) \leq N_+ $ for some $N_+ > 1$. 
Then, there is a constant $C=C(N_+, \delta_0)$ such that for any $s\in[r/2, r]$ with $ r < \min(y_0/128, \delta_0/2)$, we have:
\begin{align*}
C^{-1}s^{-3} H_{X_0}(\psi, s) \leq r^{-3} H_{X_0}(\psi, r) \leq Cs^{-3} H_{X_0}(\psi, s).
\end{align*}
\end{lemma}
\begin{proof} 
We have $\cN(t) \ge 1-C_\star t^{1/2}$ according to Lemma \ref{nprim-N}.
Additionally, we compute the derivative as follows:
\[
\begin{aligned}
\frac{d}{dt} \left( t^{-3} H(t) \right) = &\ 2 t^{-4} H(t) \left(\cN + \cV - 1 \right) - t^{-4} \int_{\partial B_{t}(X_0)}\dfrac{t-y_0}{y^2}\psi^2(X) \, d\mathcal{H}^1  \\
\ge & - 2 C_\star t^{-7/2} H(t) -  \frac{t^{-3}}{y_0-t} H(t) \\
= &  - t^{-3} H(t) \left( 2C^\star t^{-1/2} + \frac{1}{y_0-t} \right).
\end{aligned}
\]
Hence, 
$$\frac{d}{dt} \log(t^{-3}H) \ge - \left( 2C^\star t^{-1/2} + 1/(y_0-t) \right), $$
which establishes the desired lower bound in the lemma.

\medskip

For the upper bound, we use the estimate in Remark~\ref{rmrk-almost-almgren}:
\[
(y_0-r)^{-2}\cN(r) - (y_0-s)^{-2} \cN(s) \ge \int_s^r \frac{2\cV^2}{t(y_0-t)^{2} } \ dt  - 2 \left( C(\delta_0) (r-s) \int_s^r  \frac{2\cV^2}{t(y_0-t)^{2} }  \ dt\right)^{1/2}.
\]
Using  $\cN(s) \ge 1-C_\star s^{1/2}$, we obtain
\[
\int_s^r \frac{2\cV^2}{t(y_0-t)^{2} } \ dt  \le C_1(N_+, \delta_0) (1+ r).
\]
Thus,
\[
\int_{s}^{r} \dfrac{\cV(t)}{t} \, dt \leq \left( \int_{s}^{r} \dfrac{2\cV^2(t)}{t(y_0-t)^2} \, dt \right)^{1/2} \left( \int_s^r \dfrac{(y_0-t)^2}{2t}\,dt\right)^{1/2}\leq C_2 \sqrt{1 + r }.
\]
Now, recalling \eqref{almost-almgren}, we integrate the following inequlaity
\[
\frac{d}{dt} \log t^{-3} H(t) \le  \frac{2}{t} \left( \cN + \cV - 1 \right) +  \frac{1}{y_0-t}  \leq \frac{2}{t} \left( C_3(\cN(r) + r)+ \cV \right) + C_4 
\]
over the interval \([s,r]\). This yields
\[
\log r^{-3}H(r) - \log s^{-3}H(s) \leq 2C_3(\cN(r) +r)\ln \left(\frac{r}{s} \right)  + C_4(r-s)+ \int_s^r \frac{2\cV(t)}{t} \, dt \leq C_5(1 +r).
\]
Thus, we have:
\[
r^{-3}H(r) \leq s^{-3} H(s) \exp\left(C_5(1 +r)\right).
\]
\end{proof}

\begin{lemma}
\label{43}
Let \( (\psi, I) \) be a variational solution on \( \Omega \), and suppose \( B_r^{+}(X_0) \subseteq \Omega \). 
Assume that $ X_0 \in N_\psi$ with $y_0 \ge \delta_0$. 
Moreover, suppose that $M_{X_0}(\psi, 0+) =  \omega(X_0)$ and that $ \cN_{X_0}(\psi, r) \leq N_+ $ for some $N_+ > 1$. 
Then, there is a constant $C=C(N_+, \delta_0)$ such that for every $ r < \min(y_0/128, \delta_0/2)$
\[
\sup_{s \in \left[ r/2, r \right]} \cV_{X_0}(\psi, s) \leq C \cV_{X_0}(\psi, r).
\]
An analogous estimate holds when $X_0\in S_\psi$.
\end{lemma}

\begin{proof}
Note that directly from the definitions,
\[
\cV_{X_0}(\psi, s) = \frac{s \int_{B_s^+} y \left(1 - I \right) \, dX}{H_{X_0}(\psi, s)} \leq \frac{r \int_{B_r^+} y \left(1 - I \right) \, dX}{ H_{X_0}(\psi, s)} = \frac{ H_{X_0}(\psi, r)}{  H_{X_0}(\psi, s)} \cV_{X_0}(\psi, r).
\]
By Lemma \ref{42-N}, we have:
\[
\frac{ H_{X_0}(\psi, r)}{H_{X_0}(\psi, s)} \leq C\left(\frac{r}{s}\right)^{3}\le 2^{3}C; 
\]
which finishes the proof. 
\end{proof}

The next lemma essentially asserts that 
\(  \cV_{X_0}(\psi, t) \to 0 \) as \( t \to 0^+ \)
at any point \( X_0 \), where \( M_{X_0}(\psi, 0^+) =\omega(X_0) \). 
Notice that the limit $\cN_{X_0}(\psi, 0+)$ exists due to Lemmas \ref{nprim} and \ref{nprim-N}, and the condition \( \cN_{X_0}(\psi, 2t) - \cN_{X_0}(\psi, t) \to 0 \) automatically holds as \( t \to 0^+ \).

\begin{lemma} 
\label{44}
Let \( (\psi, I) \) be a variational solution on \( B_r^{+}(X_0) \subseteq \Omega \),  and let \( t \in \left( 0, r/2 \right] \). 
Assume that $ X_0 \in N_\psi$ with $y_0 \ge \delta_0$,  that $M_{X_0}(\psi, 0+) =  \omega(X_0)$, 
\[ 
\cN_{X_0}(\psi, 2t) - \cN_{X_0}(\psi, t) < \delta, 
\]
  and that $ \cN_{X_0}(\psi, r) \leq N_+ $ for some $N_+ > 1$, with  $ r < \min(y_0/128, \delta_0/2)$.
Then, there is a constant $C=C(N_+, \delta_0)$ such that 
\[
 \cV_{X_0}(\psi, t) \leq  C \sqrt{\delta + t}.
\]
Moreover, in the degenerate case $X_0\in S_\psi$, the estimate reduces to 
\[
 \cV_{X_0}(\psi, t) \leq  C \sqrt{\delta}.
\]
\end{lemma}

\begin{proof}
{\bf Degenerate Case ($X_0 \in S_{\psi}$):}
From Lemma \ref{nprim}, we have:
\[
\int_{t}^{2t} \frac{2\cV^2(s)}{s} \, ds \leq \int_{t}^{2t} \frac{d}{ds} \cN(s) \, ds = \cN(2t) - \cN(t) < \delta.
\]
Thus, there must be at least one \( s \in [t, 2t] \) such that \( \cV(s) \leq \sqrt{\delta} \).
Applying Lemma \ref{43}, we obtain:
\[
\cV(t) \leq C(N_+) \sqrt{\delta}.
\]
{\bf Non-degenerate Case ($X_0 \in N_{\psi}$):}
From the estimate in Remark \ref{rmrk-almost-almgren}, we have:
\begin{align*}
   \int_t^{2t} \frac{2\cV^2}{s(y_0-s)^{2} } \ ds  \le  & C(\delta_0) \left( \cN(2t)-\cN(t) \right) + C(N_+, \delta_0) t \\
    \le & C(N_+, \delta_0) (\delta + t),
\end{align*}
Applying again Lemma \ref{43} completes the proof.
\end{proof}

\begin{lemma}\label{46}
Let \( (\psi, I) \) be a variational solution on \( \Omega \), and suppose \( B_r^{+}(X_0) \subseteq \Omega \). 
Assume that $ X_0 \in N_\psi$ with $y_0 \ge \delta_0$. 
Moreover, suppose that $M_{X_0}(\psi, 0+) =  \omega(X_0)$ and that $ \cN_{X_0}(\psi, 2r) \leq N_+ $ for some $N_+ > 1$, with $ 3r < \min(y_0/128, \delta_0/2)$.
Then, there exists a  constant $C=C(N_+, \delta_0)$ such that for any $Y \in B_{r/4}(X_0)\cap N_\psi$, and $s < r/2$,
$$
\cN_Y(\psi,s) \leq C\cN_{X_0}(\psi,2r) + Cr \qquad \text{and} \qquad  H_Y \left(\psi,5r/8 \right) \geq C H_{X_0}(\psi,r).
$$
\end{lemma}

\begin{proof}
First, from Lemma \ref{42-N}, we have:
\[
r^{-3}\int_{\partial B_r(X_0)} \frac{1}{y} \psi^2 \, d\mathcal{H}^1 \leq C^2 s^{-3}\int_{\partial B_s} \frac{1}{y} \psi^2 \, d\mathcal{H}^1,
\]
for any \( s \in \left[ r/4, r \right] \). Integrating this gives:
\[
\int_{\partial B_r(X_0)} \frac{1}{y} \psi^2 \, d\mathcal{H}^1 \leq Cr^{-1} \int_{B_{3r/8}(X_0) \setminus B_{r/4}(X_0)} \frac{1}{y} \psi^2 \, d\mathcal{H}^1 \leq Cr^{-1} \int_{B_{3r/8}(X_0)} \frac{1}{y} \psi^2\, d\mathcal{H}^1.
\]
Taking any \( Y \in B_{r/4}(X_0) \cap N_\psi \), then $\frac{d}{dt}\log\left(t^{-3}  H_Y(t)\right) \ge -C\left(1+t^{-1/2}\right) $, and one can check that  
\[
\int_{B_{5r/8}(Y)} \frac{1}{y} \psi^2 \, dX = \int_0^{5r/8} H_Y(s) \, ds \le  \int_0^{5r/8} \left(\frac{s}{r}\right)^3\exp(Cr^{1/2})\, ds \times H_Y(5r/8) \le Cr  H_Y(5r/8).
\] 
Thus, this gives
\[
\int_{B_{3r/8}(X_0)} \frac{1}{y} \psi^2 \, dX \le \int_{B_{5r/8}(Y)} \frac{1}{y} \psi^2 \, dX  \leq Cr  H_Y(5r/8).
\] 
All these together, we imply that \( H_{X_0}(r)  \leq  C  H_Y \left( 5r/8 \right)\).

On the other hand, we have the trivial inequality:
\[
\int_{B_{5r/8}(Y)} \frac{1}{y}|\nabla \psi|^2 \, dX \le 
\int_{B_{r}(X_0)} \frac{1}{y}|\nabla \psi|^2 \, dX ,
\]
 showing that: 
\[
\cN_Y \left( 5r/8 \right) \leq C \left(\cN_{X_0}(r) + \cV_{X_0} (r)\right)\le C \left(\cN_{X_0}(r) + \sqrt{\cN_{X_0}(2r) + r}\right),
\]
where the last inequality follows from Lemma \ref{44}.
Finally, for some \( s < r/2 \),  it follows from \eqref{almost-almgren} that: 
\[
\begin{aligned}
 1 + \cN_Y(s) \leq & \left(1+Cr\right) \left(1 + \cN_Y \left( 5r/8 \right)\right)  \leq \left(1+Cr\right) \left(1 + \cN_{X_0}(r)+C\sqrt{\cN_{X_0}(2r) + r } \right) \\
\leq & \left(1+Cr\right) \left(\left(1+ Cr\right) \left(1+\cN_{X_0}(2r)\right)+C\left(\cN_{X_0}(2r) + r \right) \right) \\
\le & \left(1+Cr\right) \left(1 + C\left(\cN_{X_0}(2r) + r \right)\right).
\end{aligned}
\]
Here we used the estimate $\cN_{X_0}(2r) \geq 1 - C_\star (2r)^{1/2} > c >0$ to deduce that  $\sqrt{\cN_{X_0}(2r) + r }  \le C(\cN_{X_0}(2r) + r)$. 
Hence, $\cN_Y(s) \le C(\cN_{X_0}(2r) + r)$.
\end{proof}

\section{Analysis of frequency blow-ups}
Throughout this section, we analyze a blow-up sequence of a variational solution \((\psi_k, I_k)\) with $ \partial\{\psi>0\} \ni X_k \to X_0$  and
\begin{equation}\label{5-h1}
 M_{X_k}(\psi,0+) =  \omega(X_k).
 \end{equation}
Furthermore, consider a sequence $r_k\to 0$ and assume $\cN(2r_k) \le N_+$ for some constant $N_+$, as well as
\begin{equation}\label{5-h2}
\delta_k := \cN_{X_k}(\psi_k, 2r_k) - \cN_{X_k}(\psi_k,r_k/2) \to 0, \qquad \text{ as } k \to +\infty.
 \end{equation}
Besides, we assume that the following limit exists:
\begin{equation}\label{5-h3}
\lim_{k \to +\infty} \cN_{X_k}(\psi_k, r_k) = N_\infty.
\end{equation}
Given such a sequence, we define the re-normalized sequence
\[
v_k(X) = 
\frac{\psi_k(X_k+r_k X)}{\sqrt{r_k^{\mu(X_k)-2}H_{X_k}(\psi_k, r_k)}} ,
\]
which satisfies the property 
\[
\int_{\partial B_1^+} \frac{r_k^{\mu(X_k)-2}}{r_k^{-1}y_k+ y}\ v_k^2  \, d \mathcal{H}^1= 1.
\]

\begin{lemma}[non-degenerate case]
\label{5.1n}
Suppose that \(X_k \in N_{\psi_k} \cap \{ y \ge \delta_0 \}\), for some fixed constant \(\delta_0 > 0\), and that assumptions \eqref{5-h1}, \eqref{5-h2} and \eqref{5-h3} are satisfied.
A sub-sequence of the re-normalized sequence $v_k$ converges weakly in $W^{1,2}(B_1)$ and strongly in $L^2(B_1)$ and   $L^2(\partial B_1)$ to a function $v \in W^{1,2}(B_1)$. The function $v$ satisfies the following properties:
\begin{enumerate}
\item[(i)] $v \geq 0$ on $B_1$;
\item[(ii)] $v$ is a subharmonic  on $B_1$;
\item[(iii)] the restriction of $v$ to $B_1 \setminus B_{1/2}$ is homogeneous of order $N_{\infty} \ge 1$;
\item[(iv)] $\int_{\partial B_1}  v^2 \, d\mathcal{H}^{1} = y_0$.
\end{enumerate}
\end{lemma}

\begin{proof}
We assume that $r_k \leq \delta_0 / 256$ to ensure that all results from the previous section can be applied. 
By Lemma \ref{44}, we know that $\cV_{X_k}(\psi_k, r_k) \to 0$, so
\begin{align*}
\int_{B_1} \frac{1}{y_k + r_k y} |\nabla v_k|^2 \, dX = \frac{r_k\int_{B_{r_k}(X_k)} \frac{1}{y} |\nabla \psi_k|^2 \, dX}{H_{X_k}(\psi_k, r_k)} = \cN_{X_k}(\psi_k, r_k) + \cV_{X_k}(\psi_k, r_k) \to N_{\infty}.
\end{align*}
Given that $\int_{\partial B_1} \frac{1}{y_0} v_k^2 \, d\mathcal{H}^1 = 1$, this implies that the sequence $v_k$ is uniformly bounded in $W^{1,2}(B_1)$. 
Therefore, by Rellich’s theorem and the compact embedding on the boundary,  we can extract a subsequence that converges weakly in $W^{1,2}(B_1)$ to some $v \in W^{1,2}(B_1)$ and strongly in $L^2(B_1)$ and in $L^2(\partial B_1)$. Consequently, \textit{(iv)} holds.

Since $v_k \geq 0$, it follows that $v \geq 0$, which proves \textit{(i)}.

Let $\eta \in C^{\infty}_c(B_1^{+})$ be any test function with $\eta \geq 0$. Then,
\begin{align*}
\int_{B_1} \frac{1}{y_0} \nabla v \cdot \nabla \eta \, dX = \lim_{k \to + \infty} \int_{B_1} \frac{1}{y_k + r_k y} \nabla v_k \cdot \nabla \eta \, dX \leq 0,
\end{align*}
since the functions $\psi_k$ (and thus the functions $v_k$) are $\mathcal{L}$-subsolutions, as shown in Proposition \ref{32}. Therefore, $v$ is also a subharmonic, which proves \textit{(ii)}.

Applying now Remark \ref{r2} to $\psi_k$, and integrating from $r_k/2$ to $r_k$, we obtain:
\begin{align*}
\delta_k + O(r_k^{1/2}) = & \int_{B_{r_k} \setminus B_{r_k/2}} \frac{|X|^{-1}}{H_{X_k}(\psi_k, |X|)} \frac{1}{y} \left| \nabla \psi_k \cdot (X-X_k)  -  \psi_k \cN_{X_k}(\psi_k,|X|) \right|^2 \, dX  \\
 \le  &  \int_{B_1 \setminus B_{1/2}}  \frac{|X|^{-1}}{y_k+ r_k y} \left| \nabla v_k \cdot X  -  v_k \cN(\psi_k, r_k|X|) \right|^2 \, dX  \to 0.
\end{align*}
Since $\cN_{X_k}(\psi_k, r_k|X|) \to N_{\infty}$ uniformly in $|X|$ and $v_k$ is bounded in $W^{1,2}(B_1)$, we have:
\[
\int_{B_1\setminus B_{1/2}} \left| \nabla v_k \cdot X - v_k  N_{\infty} \right|^2 \, dX  \to 0.
\]
The functional 
$$ w \mapsto \int_{B_1 \setminus B_{1/2} }  \left|\nabla w \cdot X - w N_{\infty} \right|^2 \, dX, $$ 
is convex, and thus weakly lower semi-continuous under the weak convergence in $W^{1,2}$.
This leads to the following inequality:
$$
\int_{B_1 \setminus B_{1/2}}  \left|\nabla v \cdot X - v N_{\infty} \right|^2 \, dX \leq \lim_{k \to + \infty} \int_{B_1 \setminus B_{1/2}} \left|\nabla v_k \cdot X - v_k N_{\infty} \right|^2 \, dX =0,
$$
which implies that
\[
\nabla v \cdot X -   v N_{\infty} = 0, \qquad \text{almost everywhere on } B_1 \setminus B_{1/2}.
\]
This condition is equivalent to the order $N_{\infty}$ homogeneity, i.e., $v(rX) = r^{N_{\infty}} v(X)$ almost everywhere. 
Notice that the estimate $\cN_{X_k}(\psi_k, r_k) \ge 1 - C_\star r_k^{1/2}$ implies that  $N_\infty \ge 1$.
This completes the proof of \textit{(iii)}.
\end{proof}

 A similar result can be proved for the degenerate case by an analogous argument.
 We present the result in the following lemma without proof.

\begin{lemma}[degenerate case]
\label{5.1d}
Suppose that \(X_k \in S_{\psi_k}\) and that assumptions \eqref{5-h1}, \eqref{5-h2} and \eqref{5-h3} hold.
A sub-sequence of the re-normalized sequence $v_k$ converges weakly in $W^{1,2}_{\rW}(B_1^{+})$ and strongly in $L^2_{\rW}(B_1^{+})$ and   $L^2_{\rW}(\partial B_1^{+})$ to a function $v \in W^{1,2}_{\rW}(B_1^{+})$. The function $v$ satisfies the following properties:
\begin{enumerate}
\item[(i)] $v \geq 0$ on $B_1^{+}$;
\item[(ii)] $v$ is a $\mathcal{L}$-subsolution  on $B_1^{+}$;
\item[(iii)] the restriction of $v$ to $B_1^{+} \setminus B_{1/2}^{+}$ is homogeneous of order $N_{\infty}\ge 2$;
\item[(iv)] $\int_{\partial B_1^{+}} \frac{1}{y} v^2 \, d\mathcal{H}^{1} = 1$.
\end{enumerate}
\end{lemma}

In the above argument, it is important to note that strong convergence \( v_k \to v \) in \( W^{1,2}_{\rW}(B_1^{+}) \), or even in the annular region \( B_1^{+} \setminus B_{1/2}^{+} \), is not guaranteed. Fortunately, we do not need to fully address this challenging concentration-compactness issue. Instead, we focus on the situation where the blow-up limit is one-dimensional.

\begin{theorem}
\label{5.3}
$(i)$ Suppose that \(X_k \in S_{\psi_k}\) and that assumptions \eqref{5-h1}, \eqref{5-h2} and \eqref{5-h3} hold. 
Let $A = B_1^{+} \setminus \overline{B_{1/2}^{+}}$ and $\mathbf{n}$ be a unit vector. Furthermore, assume that the normalized sequence $(v_k)$, defined above, satisfies:
$$
\int_A \frac{1}{y} \left|\nabla v_k \cdot \mathbf{n} \right|^2 \, dX \to 0, \qquad \text{as} \quad k \to + \infty.
$$
Then, $\mathbf{n} = \pm\mathbf{e}_1$ and  $v$ is independent of the first coordinate on $A$, i.e., $v(X) = h(y)$, $N_{\infty} = 2$, and 
$$ h(y) = \alpha y^{2}, \qquad \text{in} \quad A, $$ 
for the constant $\alpha=\alpha(X_0)$ such that the normalization 
\(
\int_{\partial B_1^{+}} \frac{1}{y} v^2 \, d \mathcal{H}^1 = 1
\)
happens.

$(ii)$ Suppose  \(X_k \in N_{\psi_k}\)  and that assumptions \eqref{5-h1}, \eqref{5-h2} and \eqref{5-h3} hold. 
Let $A = B_1 \setminus \overline{B_{1/2}}$ and $\mathbf{n}$ be a unit vector. 
Furthermore, assume that the normalized sequence $(v_k)$, defined above, satisfies:
$$
\int_A \frac{1}{y_0+r_ky} \left|\nabla v_k \cdot \mathbf{n}\right|^2 \, dX \to 0, \qquad \text{as} \quad k \to + \infty.
$$
Then, on $A$, $v$ is a one dimensional solution, $N_\infty = 1$, and $v(X)=\alpha |X\cdot \tilde{\mathbf{n}}|$, where $\tilde{\mathbf{n}}$ is a unit vector perpendicular to $\mathbf{n}$.
\end{theorem}

The normalization directly determines \(\alpha\): 
\begin{equation*}
\alpha(X_0)  =\begin{cases}
\dfrac{\sqrt{3}}{2}, \quad & X_0 \in S_{\psi}, \\[8pt]
\sqrt{y_0/\pi}, \quad & X_0 \in N_{\psi}.
\end{cases}
\end{equation*}

\begin{proof}
$(i)$ We already know that $v$ is homogeneous of order $N_\infty$, and satisfies $\nabla v\cdot  \mathbf{n} =0$ in $A$. 
Hence, $v(X) = h(X\cdot \tilde{\mathbf{n}})$ where $\tilde{\mathbf{n}}$ is a unit vector perpendicular to $\mathbf{n}$, and the function $h$ must take the form 
$$ h(z) = \alpha z^{N_{\infty}}, $$
with some $\alpha\ne 0$ due to the homogeneity of $v$.
We also recall that 
\[
0 = \int_A \frac{1}{y} \left|\nabla v(X) \cdot X - N_\infty v(X) \right|^2 \, dX = \lim_{k \to + \infty} \int_A \frac{1}{y} \left|\nabla v_k(X) \cdot X - N_\infty v_k(X) \right|^2 \, dX.
\]
Expanding the square and using the weak convergence \( v_k \to v \) in \( W_{\rW}^{1,2} \), we obtain
\[
\lim_{k \to + \infty} \int_A \frac{1}{y} \left| (\nabla v_k\cdot \tilde{\mathbf{n}}) (X \cdot \tilde{\mathbf{n}}) \right|^2 \, dX = \int_A \frac{1}{y} \left|(\nabla v\cdot \tilde{\mathbf{n}}) (X \cdot \tilde{\mathbf{n}} )\right|^2 \, dX.
\]
Now, let \( U \Subset A \setminus \{ X \cdot \tilde{\mathbf{n}} = 0 \} \).  
The above convergence, together with our assumptions, implies convergence in $W^{1,2}_{\rW}$-norm.  
Therefore, we conclude that $v_k \to v$ strongly in $W^{1,2}_{\rW}(U)$.

Taking any \( \xi = (\xi_1, \xi_2) \in C_c^{0,1}(U, \mathbb{R}^2) \), this implies that
\[
\int_{U} \frac{1}{y} \left|\nabla v \right|^2 \mathrm{div\,} \xi - \frac{2}{y}\nabla v \cdot D \xi \nabla v - \frac{1}{y^2} \left|\nabla v \right|^2 \xi_2 \, dX = 0,
\]
by passing the domain variation formula for \( \psi_k \) to the limit and using that \( \cV_{X_k}(\psi, r_k) \to 0 \). Due to the one-dimensional structure of \( v \), this may be rewritten  as
\begin{align*} 
0 = &  \int_U  \frac{1}{y} \left(h'(X\cdot \tilde{\mathbf{n}}) \right)^2 \left[  \mathrm{div\,} \xi - 2 \tilde{\mathbf{n}} \cdot D\xi \tilde{\mathbf{n}} - \frac{1}{y}  \xi_2\right] \, dX \\
= & \int_U - \frac{1}{y} \nabla\left(h'(X\cdot \tilde{\mathbf{n}}) \right)^2 \cdot \xi + \left( \nabla\left( \frac{2}{y}  \left( h'(X\cdot \tilde{\mathbf{n}}) \right)^2 \right) \cdot \tilde{\mathbf{n}}\right)  \left( \xi \cdot \tilde{\mathbf{n}} \right)\, dX \\
= & \int_U  \frac{2}{y} h'(X\cdot \tilde{\mathbf{n}}) h''(X\cdot \tilde{\mathbf{n}}) (\xi \cdot \tilde{\mathbf{n}})  - \frac{2}{y^2} \left(h'(X\cdot \tilde{\mathbf{n}}) \right)^2 (\mathbf{e}_2\cdot \tilde{\mathbf{n}}) (\xi \cdot \tilde{\mathbf{n}})\, dX \\
 = & \int_U  \frac{2\alpha^2}{y^2} N_\infty^2  \left ( (N_\infty-1) y -  (\mathbf{e}_2\cdot \tilde{\mathbf{n}}) (X\cdot \tilde{\mathbf{n}}) \right) (X\cdot \tilde{\mathbf{n}})^{2N_\infty-3} (\xi \cdot \tilde{\mathbf{n}})\, dX.
\end{align*}
Thus, we must have  
\[
(N_\infty-1) y  =  (\mathbf{e}_2\cdot \tilde{\mathbf{n}}) (X\cdot \tilde{\mathbf{n}}) \quad \text{ in } U,
\]
which forces $\tilde{\mathbf{n}} =\pm \mathbf{e}_2$ and $N_\infty = 2$.

\medskip

$(ii)$
The proof is similar, but we note that the rescaled functions \( v_k \) satisfy the following domain variation formula:
\begin{align*}
\int_{U}  \left(\frac{1}{y_0+r_ky} \left|\nabla v_k \right|^2 +  (y_0+r_ky)\frac{I_k-1}{r_k^{-3}H(r_k)}\right)\mathrm{div\,} \xi & - \frac{2}{y_0+r_ky}\nabla v_k \cdot D \xi \nabla v_k \\
& - r_k \xi_2\left( \frac{1}{(y_0+r_ky)^2} \left|\nabla v \right|^2  - \frac{I_k-1}{r_k^{-3}H(r_k)}\right) \, dX = 0,
\end{align*}
After passing to the limit, we obtain:
\begin{equation}\label{Thm53-Eq}
\int_{U}  \frac{1}{y_0} \left|\nabla v \right|^2 \mathrm{div\,} \xi  - \frac{2}{y_0}\nabla v \cdot D \xi \nabla v \, dX =0.
\end{equation}
Through a similar calculation, we obtain 
\[
 \int_U 2 h'(X\cdot \tilde{\mathbf{n}}) h''(X\cdot \tilde{\mathbf{n}})  (\xi \cdot \tilde{\mathbf{n}})\, dX  = 0,
\]
which implies that $N_\infty =1 $ and $ v(X) = \alpha (X\cdot \tilde{\mathbf{n}}) $ in each component of \( A \setminus \{ X\cdot \tilde{\mathbf{n}} = 0 \} \). 
Hence, \( v(X) = \alpha_+ (X\cdot \tilde{\mathbf{n}})_+ + \alpha_- (X\cdot \tilde{\mathbf{n}})_- \) for some constants \( \alpha_\pm \).
We will now show that \( \alpha_+ = \alpha_- \), which implies that \( v(X) = \alpha |X\cdot \tilde{\mathbf{n}}| \).

We already know that $v\Delta v=0$.
Let \( \eta \in C_c^\infty(A) \) be a test function. 
Since \( v_k \) satisfies
 $\mathrm{div\,} \left(\frac{1}{y_0+r_ky}\nabla v_k \right) =0$ in $\{v_k>0\}$,  we can write
\[
0=\int_A \frac{1}{y_0+r_ky}\left(\nabla v_k\cdot \nabla \eta v_k + |\nabla v_k|^2 \eta \right) \, dX.
\]
Passing to the limit, we obtain:
\[
\lim_{k\to \infty} \int_A \frac{1}{y_0+r_ky} |\nabla v_k|^2 \eta  \, dX = - \int_A \frac{1}{y_0} \nabla v\cdot \nabla \eta v\, dX = \int_A \frac{1}{y_0} |\nabla v|^2 \eta\, dX.
\]
As above, this gives strong convergence locally in $A$. 
As a consequence, the domain variation identity \eqref{Thm53-Eq} holds for any \( \xi \in C_c^{0,1}(A, \mathbb{R}^2) \).  
Therefore, testing with a vector field supported near \( \{ X\cdot \tilde{\mathbf{n}} = 0\} \), we find:
\[
\int_{\{X \cdot \tilde{\mathbf{n}}=0\} } \left(\alpha_+^2 - \alpha_-^2 \right) (\xi \cdot \tilde{\mathbf{n}})  \, d\cH^1 = 0.
\]
Since this holds for arbitrary \( \xi \), we conclude \( \alpha_+ = \alpha_- \), and thus \( v(X) = \alpha |x| \).
\end{proof}

\section{Rectifiability of free boundary}

In this section, we apply a quantitative method from geometric measure theory, pioneered by Naber and Valtorta in \cite{zbMATH06686585}, to establish the rectifiability of the free boundary.  
This approach relies on an (almost) monotone quantity, which in our context is the Almgren frequency function, \( \cN \).
To carry out this analysis, we restrict attention to the subset  \( N_\psi^t := N_\psi \cap \{ y >t \} \), for a fixed constant \( t > 0 \), and prove the rectifiability of this set.

For a variational solution $(\psi, I)$  in $\Omega \subseteq \mathbb{R}^2$, we define the set of highest-density points as
\begin{equation}
\label{sigma_H}
\Sigma_H := \left\{X_0 \in \Omega \,: \, X_0 \in N_\psi \, \text{and} \,  M_{X_0}(\psi,0^+) = \omega(X_0) \right\}.
\end{equation}
Recall from Proposition \ref{Highest-frequency-V} that $ \omega(X_0)$ is the maximal possible value that $M_{X_0}(0^+)$ may attain at $X_0$, and from Lemma \ref{nprim-N}, the frequency function $\cN_{X_0}(r)$ is almost non-decreasing and satisfies $\cN_{X_0}(\psi, 0^+) \geq 1$.

It will be convenient to also define:
\begin{align*}
\Sigma_H^t= \left\{X_0 \in \Omega \, : \, B_t(X_0) \Subset \Omega, X_0 \in N_\psi, \, y_0>t  \, \text{ and } \, M_{X_0}(\psi,0^+) = \omega(X_0) \right\}.
\end{align*}
The sets $\Sigma_H^t$ are relatively closed in $\Omega_t= \{X_0 \in \Omega : B_t^{+}(X_0) \Subset \Omega\}$. 

We will show $\Sigma_H^t$ is rectifiable for every fix $t>0$, which in turn implies the rectifiability of $\Sigma_H$.
The proof follows from the next two lemmas.

\begin{lemma}[Quantitative splitting]
\label{63}
Let $(\psi, I)$ be a variational solution on $\Omega$. For any $\rho, \gamma, N_+>0$, there exists $\eta_0 = \eta_0(\rho, \gamma, N_+)$ such that the following holds.
Let $\eta\le \eta_0$ and $X_0\in \Sigma_H^t$  with $B_{10r}(X_0)\subset \Omega$, and assume
\[
r \le \eta, \quad  \text{ and } \quad N^{\mathrm{max}}=\max_{X \in B_{3r}(X_0) \cap \Sigma^t_H} \cN_X(\psi, 3r) \leq N_+.
\]
Then, at least one of the following alternatives holds:
\begin{enumerate}
    \item either $N^{\mathrm{max}} < 1 +\gamma$ and $\Sigma_H^t \cap B_r(X_0) \subset \{Z \in B_r(X_0): \cN_Z( \gamma r) \ge N^{\mathrm{max}} -\gamma\}$

    \item or, there is a point $Y$ such that 
    \[
    \{Z \in B_r(X_0): \cN_Z( 3\rho r) \ge N^{\mathrm{max}} -\eta\} \subset B_{\rho r}(Y).
    \]    
\end{enumerate}
\end{lemma}

\begin{remark}
Lemma \ref{63} is stronger than Assumption 0.2 in \cite{edelen2019notes}, since
\[
\cN_Z(3\eta' r) \le \cN_Z(3\rho r) + 3C_N\rho r,
\]
and it suffices to choose $\eta_0' = \eta_0 / (1 + 3C_N\rho)$ in Assumption 0.2 to ensure
\[
\{Z \in B_r(X_0): \cN_Z(3\eta' r) \ge N^{\mathrm{max}} - \eta'\} \subset \{Z \in B_r(X_0): \cN_Z(3\rho r) \ge N^{\mathrm{max}} - \eta\} \subset B_{\rho r}(Y),
\]
whenever $r < \eta'$ and $\eta = \eta'(1 + 3C_N\rho)$.
\end{remark}

\begin{definition}
Let \(\nu\) be a finite Borel measure, and define its Jones number as  
\[
\beta_{\nu}(X_0,r) = \inf \left\{ r^{-3} \int_{B_r^{+}(X_0)} d^2(X,L) \, d\nu(X) \, : \, \text{$L$ is a line in $\mathbb{R}^2$} \right\}.
\]  
The smallness of \(\beta\) indicates how closely \(\nu\) is concentrated near a line. 
\end{definition}

Next, we establish an upper bound for the $\beta$-number in terms of the frequency drop over the support of \(\nu\).

\begin{lemma}[Effective control of the $\beta$-number]
\label{61}
Let $(\psi, I)$ be a variational solution on $\Omega$, and let $X_0\in \Sigma_H^t$ with $B_{10r}(X_0)\subset \Omega$ such that $\cN(X_0) \leq N_+$. 
Let $\nu$ be a finite (positive) Borel measure.
There exist constants $C_0, \delta_0>0$, depending only on $N_+$, such that if 
\[
r \le \delta_0 \quad \text{ and } \quad \cN_{X_0}(8r) - \cN_{X_0}(\delta_0 r) < \delta_0,
\]
then
$$
\beta_{\nu}(X_0,r) \leq C_0r^{-1} \int_{B_r(X_0)} \left[ \cN_X(\psi,8r)-\cN_X(\psi,r) \right] \, d\nu(X) + C_0 \nu(B_r(X_0)).
$$
\end{lemma}

We postpone the proofs of these lemmas to the following subsections. 
The rectifiability of \(\Sigma_H^t\) will be deduced using the Naber–Valtorta method \cite{zbMATH06686585}; see also \cite{edelen2019notes}, which demonstrates that the following result follows from the two preceding lemmas combined with the almost monotonicity of \(\cN\).

\begin{theorem}
\label{72}
Let $(\psi, I)$ be a variational solution on $\Omega \supset B_1^+$.  
Then there exists a constant $C = C(t)$ such that:
\begin{enumerate}
\item For every $r\le 1$ we have the estimate 
$$ \left|B_r\left(\Sigma_H^t \cap B^{+}_{1/2} \right) \right| \leq Cr. $$

\item $\Sigma_H^t\cap B_{1/2}$ is a rectifiable curve and
\[
\mathcal{H}^{1}\left(B^{+}_{1/2} \cap \Sigma_H^t \right) \leq C.
\]

\item For $\mathcal{H}^{1}$-a.e. point $X_0 \in \Sigma_H$, we have  $\cN_{X_0}(\psi, 0^+) = 1$.
\end{enumerate}
\end{theorem}
\begin{proof}
We only need to prove part $(3)$. 
For a given constant $ \epsilon > 0 $, define the set
\[
A_\epsilon := \{X \in \Sigma_H^t \cap B^{+}_{1/2} : \cN_{X}(\psi, 0^+) > \mu(X) + \epsilon\}.
\]
We aim to show that \( \mathcal{H}^{1}(A_\epsilon) = 0 \) for every \( \epsilon > 0 \). Toward a contradiction, assume that \( \mathcal{H}^{1}(A_\epsilon) = a > 0 \). Then, for small enough \( \delta > 0 \), we have
\begin{equation} \label{theorem72-eq3}
\frac{3a}{4} \le \mathcal{H}^{1}_\delta(A_\epsilon) \le \frac{5a}{4}.
\end{equation}
Choose a covering \( \bigcup_{i} B_{r_i}(X_i) \supset A_\epsilon \) such that
\begin{equation} \label{theorem72-eq1}
\sum_{i} r_i \le \frac{3a}{2},
\end{equation}
and each \( r_i \le \min(r_0, \delta) \), where \( r_0 \) will be chosen later.

Apply Lemma \ref{63} with \( \gamma = \epsilon/2 \) and a small constant \( \rho > 0 \) (which is a dimensional constant and will be fixed later) to find \( \eta_0 = \eta_0(\rho, \epsilon) \). For any ball \( B_{r_i}(X_i) \), pick a point \( X \in A_\epsilon \cap B_{r_i}(X_i) \). Then,
\[
\cN_{X}(3r_i) \ge \cN_X(0^+) - C_N r_i \ge 1 + \epsilon - C_N \eta > 1 + \frac{\epsilon}{2},
\]
by choosing \( \eta < \eta_0 \) such that \( C_N \eta < \epsilon/2 \). In particular, set \( r_0 = \eta \), so the condition \( r_i \le \eta \) is ensured.
Therefore,
\[
N^{\mathrm{max}} := \max_{Z \in \overline{B_{3r_i}(X)} \cap \Sigma^t_H} \cN_Z(\psi, 3r_i) > 1 + \gamma,
\]
and by Lemma~\ref{63}, alternative (2) must hold. Thus, there exists a point \( Y_i \) such that
\[
E := \{Z \in A_\epsilon \cap B_{r_i}(X) : \cN_Z(3\rho r_i) \ge N^{\mathrm{max}} - \eta\} \subseteq B_{\rho r_i}(Y_i).
\]

We now cover \( A_\epsilon \cap B_{r_i}(X_i) \) by balls \( B_{\rho r_i}(Z_j) \), with \( Z_j \in A_\epsilon \cap B_{r_i}(X_i) \), so that  \( B_{\rho r_i/2}(Z_j) \) are disjoint. Define:

\noindent 
- A {\it good ball} as one intersecting \( B_{3\rho r_i}(Y_i) \),\\
- A {\it bad ball} as one disjoint from \( B_{3\rho r_i}(Y_i) \), so that for such balls
\[
\max_{Z \in \overline{B_{3\rho r_i}(Z_j)} \cap \Sigma^t_H} \cN_Z(\psi, 3\rho r_i) \le N^{\mathrm{max}} - \eta.
\]

The number of good balls is bounded by a constant \( C_g \) (depending only on dimension). The number of bad balls is at most \( C_b \rho^{-2} \), where \( C_b \) is another dimensional constant.

Let \( U_1 \) denote the collection of these balls covering \( A_\epsilon \cap B_{r_i}(X_i) \) with radius \( \rho r_i \). 
For each ball \( B_{\rho r_i}(Z_j) \) in \( U_1 \), we repeat the construction: we select a ball \( B_{\rho^2 r_i}(Y_{ji}) \) and cover \( A_\epsilon \cap B_{\rho r_i}(Z_j) \) with good and bad balls of radius \( \rho^2 r_i \). Denote this new collection by \( U_2 \).
Iterating this procedure, we obtain successive families \( U_3, U_4, \cdots \), where each collection \( U_k \) consists of balls of radius \( \tau_k = \rho^k r_i \).

Consider \( U_k \) for large \( k \), and follow the ancestry of any ball. It cannot contain more than \( N = N(\eta) \approx N^{\mathrm{max}}/\eta \) bad balls; otherwise, the frequency in the current ball would fall below zero. Thus, the number of balls in \( U_k \) is bounded by
\[
|U_k| \le (C_b \rho^{-2})^N \cdot (C_g)^k.
\]

Therefore, the 1-dimensional size of all balls in \( U_k \) is
\begin{equation} \label{theorem72-eq2}
\sum_{B_{\tau_k}(Z) \in U_k} \tau_k \le (C_b \rho^{-2})^N (C_g)^k \rho^k r_i = (C_b)^N \rho^{k - 2N} (C_g)^k r_i \le (C_b)^N \rho^{\frac{k}{2} - 2N} r_i,
\end{equation}
where we have used \( \rho \le C_g^{-2} \), a choice possible by fixing \( \rho \) sufficiently small.

We first choose \( \rho \), then apply Lemma~\ref{63} to obtain \( \eta_0(\rho, \epsilon) \), then choose \( \eta \le \eta_0 \), and finally \( r_0 = \eta \). Then, for any covering \(  \bigcup_i B_{r_i}(X_i) \supset A_\epsilon \) with \( r_i \le \min(r_0, \delta) \), we have constructed a refined covering of each \( A_\epsilon \cap B_{r_i}(X_i) \) by balls of radius \( \tau_k = \rho^k r_i \) such that their total length satisfies \eqref{theorem72-eq2}.

Now choose \( k \) large enough so that \( (C_b)^N \rho^{\frac{k}{2} - 2N} \le \frac{1}{3} \). Then, by \eqref{theorem72-eq1},
\[
\mathcal{H}^{1}_\delta(A_\epsilon) \le \sum_i \frac{r_i}{3} \le \frac{a}{2},
\]
which contradicts \eqref{theorem72-eq3}. Therefore, \( \mathcal{H}^{1}(A_\epsilon) = 0 \), and the proof is complete.
\end{proof}

\subsection{Quantitative splitting}
For the proof of the quantitative splitting property, we require the following auxiliary lemma.

\begin{lemma}
\label{6.2}
Suppose $\psi$ satisfies the Lipschitz condition \eqref{lip-condition} in $\mathbb{R}^2_+$. 
Fix points $X_0 \in \mathbb{R}^2_+$ and  $Y_0 \in B_r^{+}(X_0)$ with $r<1$, and a number $N_+ > 1$. 
Assume there are two constants $N_{Y_0}, N_{X_0} \in [1/2, N_+]$ and a number $\delta > 0$ such that
$$
\int_{B_{10r}^{+}(Y_0) \setminus B_{r}^{+}(Y_0)} \frac{1}{y} \left|\nabla \psi(X) \cdot (X-Y_0) - N_{Y_0} \psi(X) \right|^2 \, dX \leq \delta,
$$
and also
$$
\int_{B_{10r}^{+}(X_0) \setminus B_{r}^{+}(X_0)} \frac{1}{y}  \left|\nabla \psi(X) \cdot (X - X_0) - N_{X_0} \psi(X) \right|^2 \, dX \leq \delta. 
$$
Then, there is a constant $C = C(N_+)$ such that
\begin{equation}
\label{d1}
\left|N_{X_0} - N_{Y_0}\right|^2\int_{B_{10r}^{+}(X_0) \setminus B_{r}^{+}(X_0)}  \frac{1}{y}  \psi^2 \, dX \leq C \delta,
\end{equation}
and
\begin{equation}
\label{d2}
\int_{B_{9r}^{+}(X_0) \setminus B_{2r}^{+}(X_0)}  \frac{1}{y}  \left|\nabla \psi(X) \cdot (Y_0-X_0) \right|^2 \,dX \leq C \delta.
\end{equation}
\end{lemma}

\begin{proof}
For the sake of simplicity, we present the proof in the case \( r = 1 \).
We can rewrite \(Y_0 - X_0 = (X - X_0) - (X - Y_0)\) for any \(X\), leading to the identity  
\[
\nabla \psi(X) \cdot (Y_0-X_0) = \nabla \psi(X) \cdot (X-X_0) - \nabla \psi(X) \cdot ( X - Y_0).
\]  
Integrating this expression, we obtain  
\[
\begin{aligned}
\left \| \nabla \psi(X) \cdot (Y_0 - X_0) - (N_{X_0} -N_{Y_0}) \psi(X) \right\|_{L^2_{\rW} \left(B_{9}^+{(X_0)} \setminus B_2^+{(X_0)} \right)}  
& \leq  \left \| \nabla \psi(X) \cdot (X -X_0) - N_{X_0} \psi(X) \right\|_{L^2_{\rW} \left(B_{9}^{+}(X_0) \setminus B_2^{+}(X_0) \right)} \\
& \hspace{-4cm} + \left\| \nabla \psi \left(X+(X_0-Y_0) \right) \cdot (X-Y_0) - N_{Y_0} \psi(X+(X_0-Y_0)) \right\|_{L^2_{\rW} \left(B_{9}^{+}(Y_0) \setminus B_2^{+}(Y_0) \right)} \\
& \leq 2 \sqrt{\delta}.
\end{aligned}
\]
In particular, this result confirms that \eqref{d1} implies \eqref{d2}. 

For any \( X \in B_8^{+}(X_0) \setminus B_3^{+}(X_0) \), the line segment \( X + t(Y_0-X_0) \), where \( t \in [0, 1] \), remains within \( B_9^{+}(X_0) \setminus B_2^{+}(X_0) \). Consequently, we have:
\[
\begin{aligned}
\psi(X+Y_0-X_0)& e^{N_{Y_0} - N_{X_0}} - \psi(X)  
 = \int_0^1 \frac{d}{dt} \left(e^{(N_{Y_0} - N_{X_0})t} \psi(X+t(Y_0-X_0)) \right) \, dt \\
& = \int_0^1 e^{(N_{Y_0} - N_{X_0})t} \left[\nabla \psi(X+t(Y_0-X_0)) \cdot (Y_0-X_0) - (N_{X_0} - N_{Y_0}) \psi(X+t(Y_0-X_0)) \right] \, dt.
\end{aligned}
\]
We square the expression and integrate over the annulus \( B^{+}_8(X_0) \setminus B^{+}_3(X_0) \) with respect to the weighted measure \( \frac{1}{y} \, dX \):  
\begin{equation} \label{18}
\begin{aligned} 
\int_{B_8^{+}(X_0) \setminus B_3^{+}(X_0)} &  \frac{1}{y} \left|\psi(X+Y_0-X_0)e^{N_{Y_0} - N_{X_0}}-\psi(X) \right|^2 \,dX \\ 
& \leq C \int_{B_9^{+}(X_0) \setminus B_2^{+}(X_0)}  \frac{1}{y}  \left| \nabla \psi(X) \cdot (Y_0-X_0) - (N_{X_0} - N_{Y_0}) \psi(X) \right|^2 \, dX \\
& \leq c \delta.
\end{aligned}
\end{equation}
Similarly, consider the line segment \( Z_t:=X_0+t(X-X_0) \) with \( X \in \partial B_s(X_0) \) and \( t \in [1, T] \), where \( sT \leq 10 \). In this case, we obtain:
$$
\begin{aligned}
T^{-N_{X_0}} \psi(Z_T)-\psi(X)  
&= \int_1^T \frac{d}{dt} \left( t^{-N_{X_0}} \psi(Z_t) \right) \, dt \\
&= \int_1^T  t^{-N_{X_0}-1} \left[ \nabla \psi(Z_t) \cdot t(X -X_0) - N_{X_0} \psi(Z_t) \right] \, dt.
\end{aligned}
$$
By squaring and integrating over the sphere \( \partial B_s \), we obtain:  
\begin{align}
\label{12}
\int_{\partial B_s^{+}(X_0)}   \frac{1}{y}  \left| T^{-N_{X_0}} \psi(Z_T)-\psi(X) \right|^2 \, d\mathcal{H}^{1}(X) \leq C(N_+) \delta,
\end{align}
where the constant \( C(N_+) \) depends on \( T \) and \( s \) but can be chosen uniformly for \( s \in [1, 10] \). This allows us to compare the integrals of \( \psi^2 \) over spheres of different radii:
\[
\begin{aligned}
T^{-2N_{X_0}} \int_{\partial B_s^{+}(X_0)}  \frac{1}{y}  \psi^2(Z_T) \, d\mathcal{H}^{1}(X)  
& \leq C \int_{\partial B_s^{+}(X_0)}  \frac{1}{y} \psi^2(X) \,d\mathcal{H}^{1}(X) + C\delta.
\end{aligned}
\]  
Thus, we obtain:  
\[
\int_{\partial B_{Ts}^{+}(X_0)}  \frac{1}{y} \psi^2(X) \, d\mathcal{H}^{1} \leq  C \int_{\partial B_s^{+}(X_0)} \frac{1}{y} \psi^2(X) \, d\mathcal{H}^{1} + C\delta.
\]  
Similarly,  
\[
\int_{\partial B_{s}^{+}(X_0)}  \frac{1}{y}  \psi^2(X) \,d\mathcal{H}^{1} \leq  C\int_{\partial B_{Ts}^{+}(X_0)}  \frac{1}{y}  \psi^2(X) \,d\mathcal{H}^{1} + C\delta.
\]  
By averaging over annuli, we obtain:
\begin{align}
\label{17}
\frac{1}{C} \int_{B_{10}^{+}(X_0) \setminus B_{1}^{+}(X_0)}  \frac{1}{y}  \psi^2(X)dX - \delta \leq \int_{\partial B_s^{+}(X_0)}  \frac{1}{y} \psi^2(X) \, d\mathcal{H}^{1} \leq C\int_{B_{10}^{+}(X_0) \setminus B_{1}^{+}(X_0)}  \frac{1}{y} \psi^2(X) + C\delta
\end{align}
for every \( s \in [1, 10] \).

Instead, we can integrate \eqref{12} over the interval \( [3,4] \) and choose \( T = 2 \), leading to  
\begin{align}
\label{21}
\int_{B_{4}^{+}(X_0) \setminus B_{3}^{+}(X_0)}  \frac{1}{y}  \left|2^{-N_{X_0}} \psi(2X-X_0) - \psi(X) \right|^2 \, dX \leq C\delta.
\end{align}  
A similar computation, now centered at \( Y_0 \), yields  
\begin{align}
\label{22}
\int_{B_{4}^{+}(Y_0) \setminus B_{3}^{+}(Y_0)}  \frac{1}{y} \left|2^{-N_{Y_0}} \psi(2X-Y_0) - \psi(X) \right|^2 \, dX \leq C\delta,
\end{align}
as well. On the other hand, from \eqref{18}, we have the following bound by setting \( A = e^{N_{Y_0}-N_{X_0}} \), which is controlled in terms of \( N_+ \):  
\[
\int_{B_{4}^{+}(X_0) \setminus B_{3}^{+}(X_0)}  \frac{1}{y}  \left|\psi(X)-A\psi(X+Y_0-X_0) \right|^2 \, dX \leq C\delta.
\]  
Similarly, we obtain:  
\[
\int_{B_{4}^{+}(X_0) \setminus B_{3}^{+}(X_0)}  \frac{1}{y}  \left|\psi(2X-X_0)-A\psi(2X+Y_0-2X_0) \right|^2 \, dX = \frac{1}{4} \int_{B_{8}^{+}(X_0) \setminus B_{6}^{+}(X_0)}  \frac{1}{y}  \left|\psi(X)-A\psi(X+Y_0-X_0) \right|^2 \, dX \leq C\delta.
\]
Applying the triangle inequality (with all norms taken in \( L^2_{\rW}(B_4^{+}(X_0) \setminus B_3^{+}(X_0)) \)), we get:  
\[
\begin{aligned}
& \left\|2^{-N_{Y_0}} \psi(2X-X_0)-2^{-N_{X_0}} \psi(2X-X_0) \right\|   \leq \left\|2^{-N_{X_0}} \psi(2X-X_0) -\psi(X) \right\| + \left\|\psi(X) - A\psi(X+Y_0-X_0) \right\| \\
& \qquad + \left\|A\psi(X+Y_0-X_0)-A2^{-N_{Y_0}} \psi(2X+Y_0-2X_0) \right\| + \left\|A2^{-N_{Y_0}}\psi(2X+Y_0-2X_0)-2^{-N_{Y_0}} \psi(2X-X_0) \right\| \\
& \hspace{165pt} \leq C \sqrt{\delta},
\end{aligned}
\]  
where we have used \eqref{21} and \eqref{22} to bound the first and third terms, respectively. Rewriting, we obtain:  
\[
\left|2^{-N_{Y_0}}-2^{-N_{X_0}} \right|^2 \int_{B_{8}^{+}(X_0) \setminus B_{6}^{+}(X_0)}  \frac{1}{y}  \psi^2(X) \, dX \leq C \delta.
\]  
Since \( \left|2^{-N_{Y_0}}-2^{-N_{X_0}}\right| \geq c(N_+) \left|N_{Y_0} - N_{X_0} \right| \), using \eqref{17}, we conclude:  
\[
\left|N_{Y_0} - N_{X_0} \right|^2 \int_{B_{10}^{+}(X_0) \setminus B_{1}^{+}(X_0)}  \frac{1}{y} \psi^2(X) \, dX \leq C \delta.
\]  
This completes the proof of \eqref{d1}.
\end{proof}

\begin{proof}[Proof of Lemma \ref{63}]
We assume that neither alternative holds for a sequence $(\psi_k, I_k)$ of variational solutions and numbers $\eta_k \to 0$, and proceed to derive a contradiction. We may assume that there are points $Z^1_k, Z^2_k, Y_k \in \Sigma_H \cap B_{r_k}(X_0)$ for $r_k \le \eta_k$ and $N^{\mathrm{m}}_k\le N_+$ such that
\begin{equation}\label{lem63-eq1}
|Z^1_k - Z^2_k| \ge \rho r_k, \quad
\cN_{Z^1_k}(\psi_k, 3\rho r_k) \ge N^{\mathrm{m}}_k -\eta_k, \quad \cN_{Z^2_k}(\psi_k, 3\rho r_k) \ge N^{\mathrm{m}}_k -\eta_k, 
\end{equation}
and also
\begin{equation}\label{lem63-eq3}
N^{\mathrm{m}}_k > 1 + \gamma \quad \text{ or } \quad\cN_{Y_k}(\psi_k, \gamma r_k) < N^{\mathrm{m}}_k - \gamma.
\end{equation}

Integrating the formula in Remark \ref{r2} and using Remark \ref{rmrk-almost-almgren} given us
$$
\int_{3\rho r_k}^{3r_k} \int_{\partial B_s(Z_k^i)}  \frac{2s^{-1}}{H_{Z_k^i}(\psi_k, s)} \frac{1}{y}  \left| \nabla \psi_k(X) \cdot \left(X-Z_k^i\right) - \cN_{Z_k^i}( s) \psi_k(X)  \right|^2 \, d\mathcal{H}^{1}ds  \le  \left(\cN_{Z_k^i}(3r_k) - \cN_{Z_k^i}(3\rho r_k) + C r_k^{1/2}(1+  N^{\mathrm{m}}_k)\right).
$$
Using the inequality 
\begin{equation}\label{lem63-eq2}
-Cr_k (1+  N^{\mathrm{m}}_k) -\eta_k + N^{\mathrm{m}}_k \le -Cr_k (1+ \cN_{Z_k^i}(3r_k))  + \cN_{Z_k^i}(3\rho r_k) \le \cN_{Z_k^i}(s) \le \cN_{Z_k^i}(3r_k) + Cr_k (1+ \cN_{Z_k^i}(3r_k))  \le N^{\mathrm{m}}_k + Cr_k  (1+  N^{\mathrm{m}}_k) ,
\end{equation}
and $H_{Z_k^i}(\psi_k, s) \le C H_{Z_k^i}(\psi_k, 3r_k)$, we obtain: 
\[
\frac{r_k^{-1}}{H_{Z_k^i}(\psi_k, 3r_k)}\int_{B_{3r_k}(Z_k^i)\setminus B_{3\rho r_k}(Z_k^i)}  \frac{1}{y} \left| \nabla \psi_k(X) \cdot \left(X-Z_k^i\right) - N^{\mathrm{m}}_k \psi_k(X)  \right|^2 \, dX = O(\eta_k),
\]
where the ratio $O(\eta_k)/\eta_k $ is bounded by a constant depending only on $t$ and $N_+$.
Now apply Lemma \ref{46}, to find the inequality  $H_{Z_k^i}(\psi_k, 3r_k) \le C H_{Z_k^1}(\psi_k, 2r_k)$ for $i=1,2$.
Next, apply  Lemma \ref{6.2}, centered at $Z_k^1$ and using $Z_k^2$, to get
\[
\frac{r_k^{-1}}{H_{Z_k^1}(\psi_k, 2r_k)}\int_{B_{2r_k}(Z_k^1)\setminus B_{r_k}(Z_k^1)}  \frac{1}{y} \left| \nabla \psi_k(X) \cdot \left(Z_k^1-Z_k^2\right) \right|^2 \, dX = O(\eta_k).
\]
From \eqref{lem63-eq1}, we may assume that $(Z_k^1-Z_k^2)/r_k \to \mathbf{n} \ne0$.
Consider the sequence of re-normalized solutions 
$$ v_k(X) = \frac{\psi_k(Z_k^1 + 2r_kX)}{\sqrt{(2r_k)^{-1}H_{Z_k^1}(\psi_k, 2r_k)}}. $$
These satisfy the hypotheses of Lemma \ref{5.1n} and Theorem \ref{5.3} (notice that $\nabla v_k$ is bounded in $L^2_{\rW}(B_1^+)$), so in particular, 
$$\lim_{k\to\infty}\cN_{Z_k^i}(2r_k) = 1.$$ 
From \eqref{lem63-eq2}, we obtain that $N^{\mathrm{m}}_k\to 1$, which contradicts the assumption $N^{\mathrm{m}}_k > 1+\gamma$ in \eqref{lem63-eq3}. 
Moreover, from the second condition in \eqref{lem63-eq3}, we have
$\cN_{Y_k}(\psi_k, \gamma r_k) < 1-\gamma/2$ for sufficiently large $k$. 
This contradicts  the lower bound $\cN_{Y_k}(\psi_k, \gamma r_k) \ge 1- C_\star(\gamma r_k)^{1/2}$; see Lemma \ref{nprim-N}.
\end{proof}

\subsection{Effective control} 
\begin{proof}[Proof of Lemma \ref{61}]

Applying Lemmas \ref{42-N} and \ref{46}, we obtain that for any \( Y \in \Sigma_H^t \cap \overline{B_r(X_0)} \), we have:  
\[
\cN_{Y}(9r) \leq C (1 + N_+).
\]  
Moreover, from the second conclusion in Lemma \ref{46}, combined with repeated applications of Lemma \ref{42-N}, we deduce for all $s\in[r, 9r]$:  
\[
\frac{1}{C} H_{X_0}(10r) \leq  H_{Y}(s) \leq C  H_{X_0}(10r).
\]  
Next, we select  a point \( Y_0 \in \Sigma_H^t \cap \overline{B_r(X_0)} \) such that  
\[
\cN_{Y_0}(8r) - \cN_{Y_0}(  r) = \min_{X \in \Sigma_H^t \cap \overline{B_r(X_0)}} \left[ \cN_X(8r) - \cN_X( r) \right].
\]  
Since the function \( \cN_X(r) \) is continuous in \( X \), and the set \( \Sigma_H^t \cap \overline{B_r(X_0)} \) is closed, the minimum is attained. 
Moreover, by the assumption in the statement of the lemma  for a sufficiently small constant \( \delta_0 = \delta_0(N_+) \), to be chosen later, we have 
\begin{equation}\label{lem61-eq1}
    \cN_{Y_0}(8r) - \cN_{Y_0}( r) \leq  \cN_{X_0}(8r) - \cN_{X_0}( r) \le  \cN_{X_0}(8r) - \cN_{X_0}(\delta_0 r) + C(N_+)r  \le \delta_0 + C(N_+)r .
\end{equation}

At any point \(Y \in \Sigma_H^t \cap \overline{B_r(X_0)}\), integrating the formula in Remark \ref{r2}, we obtain
$$
\int_{r}^{8r} \int_{\partial B_s(Y)} \frac{2s^{-1}}{H_{Y}( s)} \frac{1}{y} \left| \nabla \psi(X) \cdot \left(X-Y\right) - \cN_{Y}( s) \psi(X)  \right|^2 \, d\mathcal{H}^{1}(X) \leq \cN_{Y}(8r) - \cN_{Y}(r) + C(N_+) r^{1/2}.
$$
From Remark \ref{rmrk-almost-almgren}, we also have
\[ 
-Cr (1+\cN(s) ) \le \cN_{Y}(s) - \cN_{Y}(r) \leq \cN_{Y}(8r) - \cN_{Y}(r) + Cr ( 1+ \cN(8r) ), 
\]
which implies
\[
|\cN_{Y}(s) - \cN_{Y}(r)| \leq  \cN_{Y}(8r) - \cN_Y( r)  + C(N_+) r . 
\]
Combining this with inequalities \( H_{Y}(s) \leq CH_{X_0}(10r) \leq CH_{Y_0}(6r) \),  we deduce
\begin{align*}
\int_{B_{8r}(Y)\setminus B_{r}(Y)}  \frac{1}{y} \left| \nabla \psi(X) \cdot \left(X-Y \right) - \cN_{Y}(r) \psi(X) \right|^2 \, dX \leq CrH_{Y_0}(6r) \left[\cN_{Y}(8r) - \cN_{Y}(r)+ C(N_+) r  \right].
\end{align*}
Applying this estimate at both \(Y_0\) and \(Y\) into Lemma \ref{6.2}, equation \eqref{d2} leads to:  
\begin{align}
\label{6.21}
\int_{B_{6r}(Y_0) \setminus B_{2r}(Y_0)}  \frac{1}{y} \left|\nabla \psi(X) \cdot \left(Y-Y_0\right) \right|^2 \, dX \leq CrH_{Y_0}(6r) \left[\cN_{Y}(8r) - \cN_{Y}( r) + C(N_+) r \right].
\end{align}
Here, we have used the fact that \(\cN_{Y_0}(8r) - \cN_{Y_0}(\delta_0r) \leq \cN_{Y}(8r) - \cN_{Y}(\delta_0r) \) .

Now, considering a line
\[
L = \{X: (X - Y_0) \cdot \boldsymbol{e} = 0\},
\]  
the definition of \(\beta = \beta_{\nu}(X_0,r)\) provides:  
\begin{align*}
\beta r^{3} \leq \int_{B_r(X_0)}  \left| \left(X - Y_0 \right) \cdot \boldsymbol{e} \right|^2 \, d\nu(X).
\end{align*}
Substituting \(\boldsymbol{e} = \frac{\nabla \psi(\tilde{X})}{|\nabla \psi(\tilde{X})|}\) for any point \(\tilde{X} \in \{\psi > 0\}\), we obtain:  
\begin{align*}
\beta r^{3}\left|\nabla \psi(\tilde{X}) \right|^2 \leq \int_{B_r(X_0)}  \left|\nabla \psi(\tilde{X}) \cdot \left(X-Y_0 \right) \right|^2 \, d\nu(X).
\end{align*}
Integrating over \(A = \left(B_{6r}(Y_0) \setminus B_{2r}(Y_0) \right) \cap \{\psi > 0\}\) with respect to singular measure, and switching the order of integration, we arrive at:  
\begin{align*}
\beta r^{3}\int_{A} \frac{1}{\tilde{y}} \left|\nabla u(\tilde{X}) \right|^2 \, d\tilde{X} \leq \int_{B_r}  \int_{A} \frac{1}{\tilde{y}} \left| \nabla \psi(\tilde{X}) \cdot (X-Y_0) \right|^2 \, d\tilde{X} d\nu(X).
\end{align*}
Using \eqref{6.21}, we deduce:  
\begin{align*}
\beta r^{3}\int_{A} \frac{1}{\tilde{y}} \left| \nabla \psi(\tilde{X}) \right|^2 \, d\tilde{X} \leq CrH_{Y_0}(6r) \int_{B_r^{+}}   \left[ \cN_X(8r) - \cN_X(r) + C(N_+) r \right] \, d\nu(X).
\end{align*}
Now, to complete the proof, it remains to show that 
\begin{equation}\label{lem61-eq2}
\frac{r}{H_{Y_0}(6r)} \int_{A} \frac{1}{y} \left|\nabla \psi(X) \right|^2 \, dX \geq c > 0.
\end{equation}
We will prove that there exist constants $\delta_0$ and $c>0$ such that if \eqref{lem61-eq1} holds for $r<\delta_0$, then the inequality \eqref{lem61-eq2} also holds.

Toward a contradiction, suppose that there exist sequences $\psi_k$, $Y_k$ and $r_k< \frac{1}{k}$ such that
\[
 \cN_{Y_k}(\psi_k, 8r_k) - \cN_{Y_k}(\psi_k,r_k) \leq \frac{C}{k} \quad \text{ and } \quad 
 \frac{r_k}{H_{Y_0}(\psi_k,6r_k)}\int_{A} \frac{1}{y} \left|\nabla \psi_k(X) \right|^2 \, dX \le \frac{1}{k}.
\]
Furthermore, we may assume that $ \cN_{Y_k}(\psi_k, 6r_k)\to N_\infty$ with \( N_\infty \ge 1 \), since \( \cN_{Y_k}(\psi_k, 6r_k) \ge 1 - C_\star (6r_k)^{1/2} \).
Now consider the sequence of re-normalized
solutions
\[
v_k(X) = \frac{\psi_k(Y_k + 6r_kX)}{\sqrt{(6r_k)^{-1}H_{Y_k}(\psi_k, 6r_k)}},
\]
and apply Lemma \ref{5.1n}. 
Therefore, $v_k\to v$ where the limit function $v$ is homogeneous of degree $N_\infty \ge 1$ in $ B_1 \setminus B_{1/2}$. 
However, passing to the limit yields
\[
\int_{B_1\setminus B_{1/3}} \frac{1}{y} \left|\nabla v(X) \right|^2 \, dX =0,
\]
which implies $v$ is non-zero constant in $ B_1 \setminus B_{1/2}$; recall $(iv)$ from Lemma \ref{5.1n}. 
This contradicts the homogeneity of $v$ of order greater than one.
\end{proof}

\section{A complete identification of the measure \texorpdfstring{$\mathcal{L}u$}{Lu} of variational solutions}

Theorem \ref{72} allows us to fully characterize $\mathcal{L}u$ for variational solutions. In particular, Theorem \ref{MAIN-THEOREM} (iii) states that, in the sense of distributions, a variational solution $(\psi, I)$ satisfies:
$$
\mathcal{L}\psi = \mathrm{div}\left(\frac{1}{y} \nabla \psi \right) =  \mathcal{H}^{1}\llcorner \partial^{\star}\{\psi>0\} + \frac{2}{\sqrt{\pi y}} \sqrt{\lim_{r \to 0^+} r^{-3}H_{X}(\psi,r)}\mathcal{H}^{1} \llcorner\Sigma_{H}.
$$
To establish this identification of the measure $\mathcal{L}\psi$, we require an analysis of the regular free boundary (subsection \ref{8.1}), an examination of singular points (subsection \ref{8.2}), and the identification of the function \( I \) for a variational solution $(\psi, I)$ (subsection \ref{8.3}).

\subsection{Analysis of regular points}
\label{8.1} 
Let $(\psi, I)$ be a variational solution in $\Omega$. Due to the upper semi-continuity of the mapping \(X_0 \mapsto M_{X_0}(\psi, 0^+)\) (see Proposition \ref{Highest-frequency-V}), the set $\Sigma_H$ is relatively closed in $\Omega$. Consequently, $\Omega \setminus \Sigma_H$ is an open set. In this section, we review some results concerning $\Omega \setminus \Sigma_H$, particularly focusing on the portion of the free boundary $\partial\{\psi> 0\}$ contained within it. Specifically, this corresponds to the part of the free boundary where the density satisfies \(M_{X_0}(0^+) <  \omega(X_0) \).

To begin with, observe that within the interior of \(\{\psi = 0\}\), the function \(I\) must be locally constant. Indeed, equation \eqref{inner-variation-1} yields  
\[
0=\int y I \,\mathrm{div} \eta +  \eta_2 I \, dX =  \int  I \,\mathrm{div} (y\eta)  \, dX,
\]  
for all \(\eta = (\eta_1, \eta_2) \in C^{\infty}_c(\{\psi = 0\}^{\circ}, \mathbb{R}^2)\) such that \(\eta_2 = 0\) on \(\{y=0\}\), implying that \(I\) is constant within each connected component in $\{y>0\}$.
If \(I = 0\) in a given component, we obtain \(M_{X_0}(\psi, 0^+) \equiv 0\) there due to Proposition \ref{Highest-frequency-V}. 
Conversely, if \(I = 1\), we encounter more complicated case where \(M_{X_0}(\psi, 0^+) \equiv  \omega(X_0) \). The latter scenario is certainly possible, as the pair \((\psi, I) = (0,1)\) constitutes a variational solution. 
However, since the mapping \(X_0 \mapsto M_{X_0}(\psi, 0^+)\) is upper semi-continuous, it follows that at any point \(X_0 \in \partial (\{\psi = 0\}^{\circ})\), 
we must have \(M_{X_0}(\psi, 0^+) =  \omega(X_0) \), which implies \(X_0 \in \Sigma_H\).  
In Subsection \ref{8.2}, we will establish that this scenario is actually impossible. Consequently, \(I = 0\) within \(\{\psi = 0\}^{\circ}\), unless \(\psi\) is identically zero. In this subsection, we simply note that within a connected open set where \(\Sigma_H = \emptyset\), unless \(\psi \equiv 0\), we necessarily have \(M_{X_0}(\psi, 0^+) < \omega(X_0)\).

\begin{lemma}
\label{81}
Let $(\psi, I)$ be a variational solution in $\Omega$, and $B^{+}_{2r}(X_0) \Subset \Omega$ with $y_0>0$. Then, if $\psi(X_0) = 0$,
\begin{align*}
M_{X_0}(\psi,0^+) \in \{0\} \cup \left[ \frac{\omega(X_0)}{2}, \omega(X_0) \right],
\end{align*}
with $M_{X_0}(\psi, 0^+) = 0$, only if $X_0$ is in the interior of $\{\psi = 0\}$. Moreover, if $X_0 \in N_\psi$ and the
blow-up sequence $$\psi_r(X):=\frac{\psi(X_0 + r X)}{r} \to 0, \qquad \text{as $r \to 0$}; $$ 
then $M_{X_0}(\psi, 0^+) = \omega(X_0)$.
\end{lemma}

\begin{proof}
Consider a blow-ups of \(\psi\) along a sub-sequence that converges to a homogeneous function \(\psi_0\) of degree one; Proposition \ref{10.1}.
First, note that if \( \psi_0 \equiv 0 \), the limit of \(I(X_0 + rX)\) as \(r \to 0\) must be a constant function. 
Hence,  \( M_{X_0}(\psi, 0^+) \in \{0, \, \omega(X_0)\} \).

Now we establish that \(M_{X_0}(\psi, 0^+) = 0\) can occur only if \(\psi_0 \equiv 0\) (along each convergent sub-sequence), and \(M_{X_0}(\psi, 0^+) \geq \omega(X_0)/2\) otherwise. 
According to Proposition \ref{10.1},
if \(\psi_0\) is nonzero, then 
\[
\left|\{\psi_0 > 0\} \cap B_1 \right| \geq \frac{ \omega(X_0)}{2y_0},
\]  
and thus
\[
\frac{ \omega(X_0)}{2} \le  y_0 \left|\{\psi_0 > 0\} \cap B_1 \right| \leq \liminf_{k \to + \infty} \frac{\left|\{ \psi_{r_k}>0\}\cap B_{r_k}\right|_{\rW}}{r_k^2} \leq \lim_{k \to + \infty} \frac{ \left|\{I_k=1\}\right|_{\rW}}{r_k^2} = M_{X_0}(\psi, 0^+).
\]
We conclude that \( M_{X_0}(\psi, 0^+) = 0 \) can occur only if \( \psi_0 \equiv 0 \).

The function \(X \mapsto M_{X}(\psi, 0^+)\) is upper semi-continuous on \(\Omega\) (see Proposition \ref{Highest-frequency-V}). 
Therefore,  
\[
\left\{ X \,:\, M_{X}(\psi, 0^+) \leq 0 \right\}\cap \{y>0\} = \left\{X \,:\, M_{X}(\psi, 0^+) <  \omega(X)/4 \right\}\cap\{y>0\}
\]  
is relatively open. 
Let \(X_0\) be a point such that \(M_{X_0}(\psi, 0^+) = 0\) with $y_0>0$, and assume that \(M_{X}(\psi, 0^+) \leq 0\) for all  \(X\in B_{\delta}(X_0)\) with some small \(0< \delta < y_0\). 
Then, at every point \(X \in B_{\delta}(X_0)\), we must either have \(M_X(\psi, 0^+) = 0\), and hence  
\[
\lim_{s \to 0} \frac{\left|\{I=1\}\cap B_s(X)\right|_{\rW}}{s^2} = 0,
\]  
or \(\psi(X) > 0\), and so  
\[
\lim_{s \to 0} \frac{\left|\{I=1\}\cap B_s(X)\right|_{\rW}}{s^2} = \omega(X).
\]
Therefore, the density of $\{I=1\}$ is zero or one.
This, however, implies that the essential boundary of \(\{I = 1\}\) is empty in \(B_{\delta}(X_0)\), which  yields that either \(I \equiv 0\) or  \(I \equiv 1\) Lebesgue-a.e. on \(B_{\delta}(X_0)\) (indeed, Lemma \ref{BV-estimate} implies that $\{I = 1\}$ is a set of finite perimeter; \cite[Lemma 5.3]{evans2018measure} shows that its reduced boundary is trivial and so the perimeter measure of $\{I=1\}$; and the isoperimetric inequality \cite[Theorem 5.11]{evans2018measure} then yields the claim).
The later case is impossible because, at least at one point, \(X_0\), the density is zero, meaning that \(I \equiv 0\). 
This implies that \(M_X(\psi, 0^+) = 0\) and \(\psi(X) = 0\) for Lebesgue-a.e. \(X \in B_{\delta}(X_0)\), and by continuity, \(\psi = 0\) on all of \(B_{\delta}(X_0)\).

The final statement follows from the fact that  \( M_{X_0}(\psi, 0^+) \in \{0, \, \omega(X_0)\} \) and as shown earlier in the lemma, \(M_{X_0}(\psi, 0^+) \neq 0\) when $X_0 \in \partial\{\psi>0\}$.
\end{proof}

\begin{lemma}
\label{82}
Let \((\psi, I)\) be a variational solution on a connected open set \(\Omega\), and assume that \(\Sigma_H = \emptyset\). Then, either \(\psi \equiv 0\) or \(I = \chi_{\{\psi > 0\}}\) Lebesgue-a.e., implying that the set \(\{\psi > 0\}\) has locally finite perimeter. Furthermore,  
\[
\mathcal{H}^{1} \left(B^{+}_r(X_0) \cap N_\psi \right) \leq C(C_V) r
\]  
for any \(B^{+}_{2r}(X_0) \Subset \Omega\), where \(C_V\) is the constant in the definition of the variational solution.
\end{lemma}

\begin{proof}
As discussed at the beginning of this section, we have that \( I = \chi_{\{\psi > 0\}} \) on \( \Omega \setminus \partial \{\psi > 0\} \), and on \(N_\psi \), we have \( M_{X_0}(0^+) \in [\omega(X_0) / 2, \omega(X_0)) \), due to Lemma \ref{81}. 
In particular, on \( N_\psi \), the Lebesgue density of \( \{I = 1\} \) lies in \( [1/2, 1) \), so \( N_\psi \) is contained in the essential boundary \( \partial^e \{I = 1\} \). Indeed,
\[
\begin{aligned}
M_{X_0}(0^+) = & \lim_{s\to 0} \frac{\left|\{I=1\}\cap B_s(X_0)\right|_{\rW}}{s^2} = \lim_{s\to 0} \frac{1}{s^2}\int_{B_s(X_0)} yI\, dX\\
\le & \lim_{s\to 0} \frac{s+y_0}{s^2}\int_{B_s(X_0)} I\, dX = y_0 \lim_{s\to 0} \frac{\left|\{I=1\}\cap B_s(X_0)\right|}{s^2}.
\end{aligned}
\]
From Lemma \ref{BV-estimate}, this implies that
\[
\mathcal{H}^{1} \left(B^{+}_r(X_0) \cap N_\psi \right) \leq \mathcal{H}^{1} \left(B^{+}_r(X_0) \cap \partial^e \{I = 1\} \right) = \int_{B^{+}_r(X_0)} \left|\nabla \left(y I - y \right) \right| \leq C(C_V) r + r^2
\]
when \( B^{+}_{2r}(X_0) \Subset \Omega \). 
This implies that \( | N_\psi| = 0 \), \( I = \chi_{\{\psi > 0\}} \) Lebesgue-a.e., and that \( \{\psi > 0\} \) has finite perimeter.
\end{proof}

To proceed, it will be useful to introduce the notion of a viscosity solution for our problem.

\begin{definition}[Viscosity solution]
Let \(\psi\) be a continuous function \(\psi : \Omega \to [0, +\infty)\) that is an \(\mathcal{L}\)-solution in \(\{\psi > 0\}\). Then, \(\psi\) is a viscosity solution if, for any smooth function \(\phi\) satisfying either \(\phi \leq \psi\) or \(\phi_+ \geq \psi\) in \(B^{+}_r(X_0) \subseteq \Omega\), with \(X_0 \in N_{\psi}\) and \(\phi(X_0) = 0\), the following conditions hold: 

\begin{enumerate}
\item[(i)] \(\phi \leq \psi\) implies \( |\nabla \phi(X_0)| \leq y_0 \).  
\item[(ii)]  \(\phi_+ \geq \psi\) implies \( |\nabla \phi(X_0)| \geq y_0 \).
\end{enumerate}
\end{definition}

\begin{lemma}
Let $(\psi, I)$ be a variational solution on $\Omega$, and assume $\Sigma_H=\emptyset$. Then, $\psi$ is a viscosity
solution on $\Omega$.
\end{lemma}

\begin{proof}
Let $\phi \leq \psi$ (or $\phi_+ \geq \psi$) in $B^{+}_r(X_0)$, where $\phi(X_0) = \psi(X_0) = 0$ and $X_0 \in N_\psi$. 
Without loss of generality, we assume that if the first inequality holds, then $|\nabla \phi(X_0)| \neq 0$. 
The condition $X_0 \in N_\psi$ ensures, by Lemma \ref{81}, that $M_{X_0}(\psi, 0^+) \geq  \omega(X_0) / 2$. 
However, since $\Sigma_H = \emptyset$, it follows that $M_{X_0}(\psi, 0^+) < \omega(X_0)$.
Then by Proposition \ref{10.1}, along a sequence $s_k$, we obtain
$$
\psi_k(X) = \frac{\psi(X_0 + s_k X)}{s_k} \;\to\; \psi_0(X) = y_0 (X\cdot \nu)_+
$$
locally uniformly and strongly in $W_{\mathrm{loc}}^{1,2}$, where $\nu$ is a unit vector. 
On the other hand, we have
\begin{equation}
\label{diff-cases}
\lim_{s_k \to 0^+} \frac{\phi(X_0 + s_k X)}{s_k} =
\nabla \phi(X_0) \cdot X.
\end{equation}
By the homogeneity of $\psi_0$, it follows that $\nabla \phi(X_0) \cdot X \leq \psi_0(X)$ (or $(\nabla \phi(X_0) \cdot X)_+ \geq \psi_0(X)$). 
Therefore, $\nabla \phi(X_0)$ is parallel to $\nu$ and $|\nabla \phi(X_0)| \le y_0$ (or $|\nabla \phi(X_0)| \ge y_0$).
\end{proof}

\begin{corollary}
\label{84}
Let $(\psi, I)$ be a variational solution on a connected open set $\Omega$, and assume that $\Sigma_H=\emptyset$.  
Then, the reduced boundary $\partial^{\star}\{\psi> 0\}$ is identical to $\partial\{\psi> 0\}$ and is locally represented by graphs of analytic functions, along which $|\nabla \psi| = y$ from the $\{\psi> 0\}$ side. 
Moreover, $\psi$ satisfies  
\begin{align*}
\mathcal{L}\psi =  \mathcal{H}^{1} \llcorner \partial^{\star}\{ \psi>0\},
\end{align*}  
in the sense of distributions.
\end{corollary}

\begin{proof}
First note that by $\epsilon$-flatness the free boundary of a viscosity solution $\psi$
in $B^{+}_1$ in the sense
\begin{align*}
\left(x - \epsilon\right)^+ \leq \psi(X) \leq (x + \epsilon)^+
\end{align*}
implies the free boundary to be $C^{1,\alpha}$ in $B^{+}_{1/2}$; see e.g. the developed theory in \cite{zbMATH05934884} or \cite{du2024regularity, du2023free2}.
This, in turn, implies that the reduced boundary $\partial^{\star}\{\psi>0\}$, which coincides with the topological free boundary $\partial \{\psi>0\}$, is locally given by the graph of an analytic function, according to well-established theory (see, e.g., \cite{MR752578}).
Next we observe that by the assumption $\Sigma_H = \emptyset$ as well as Lemma \ref{81}, a limit of
the blow-up family $\psi_r(X) = \psi(X_0+rX)/r$ as a sequence $r\to 0$ at free boundary points $X_0$ cannot be zero. 
Therefore, by  Proposition \ref{10.1} we obtain the blowup solution must be $\psi_0(X) = y_0 (X\cdot \nu)_+ $ for some unit vector $\nu$. 
Therefore, for sufficiently small $r$, the $\epsilon$-flatness condition holds for $\psi_r$.
\end{proof}

\subsection{Analysis of singular points}
\label{8.2}

The analysis of regular points does not rely on Theorem \ref{72}. However, in the presence of singular points, the absence of this theorem leads to significant challenges: the set $\Sigma_H$, and consequently the support of $\mathcal{L}\psi$, could be large and exhibit a highly irregular structure. Theorem \ref{72} resolves the most major irregularities, allowing for more accurate characterization of $\mathcal{L}\psi$ and $I$.

\begin{lemma}
\label{85}
Let $(\psi, I)$ be a variational solution in $\Omega$. Then, for $\mathcal{H}^{1}$-a.e. $X_0 \in \Sigma_H$, the following hold:  
\begin{enumerate}
\item The renormalized functions  
\begin{align*}
v_{X_0,r}(X) = \frac{\psi(X_0+rX)}{ \sqrt{r^{-1}H_{X_0}(\psi,r)}} \to \alpha(X_0) |X\cdot \nu|,
\end{align*}  
converge strongly in $W^{1,2}_{\rW, \mathrm{loc}}(\mathbb{R}^2)$, where $\nu$ is a unit vector.  
\item The Lebesgue density of $\{\psi> 0\}$ at $X_0$ exists and equals $1$:  
\begin{align*}
\lim_{r \to 0} \frac{\left|B_r(X_0) \cap \{\psi>0\}\right|}{\left|B_r(X_0)\right|}=1.
\end{align*}  

\end{enumerate}
\end{lemma}

\begin{proof}
From Theorem \ref{72}, for $\mathcal{H}^{1}$-a.e. \( X_0 \in \Sigma_H \), we have \( \cN_{X_0}(\psi, 0^+) = 1 \), and \(\Sigma_H\) possesses an approximate tangent. That is, the measures  
\begin{align*}
\boldsymbol{\nu}_{X_0,r}(E) = \frac{\mathcal{H}^{1}\left(\left(X_0+rE\right)  \cap \Sigma_H \right)}{r^{2}}
\end{align*}  
converge in the weak-* sense to \(\mathcal{H}^{1} \llcorner L\) for some line \( L \subseteq \mathbb{R}^2 \) passing through the origin.  

We now show that at any such point \( X_0 \), both \textit{(1)} and \textit{(2)} hold. 
By the continuity of the mapping \( X \mapsto \cN_X(\psi, r) \) for each small positive \( r \), we may assume that for every \( \eta > 0 \), there exists a \( \delta > 0 \) such that 
\[
\cN_X(\psi, r) \leq 1 + \eta \qquad \text{for} \quad X \in B_{\delta} \cap \Sigma_H \text{ and } r < \delta.
\]
Given any sequence \( r_k \to 0 \), consider the renormalized sequence  
\begin{align*}
v_k(X)  = \frac{\psi(X_0+r_k X)}{ \sqrt{r_k^{-1} H_{X_0}(\psi, r_k)}}.
\end{align*}  
By  Lemma \ref{5.1n}, the sequence \( v_k \) converges weakly to \( v \) in \( W^{1,2}_{\rW}(B_R) \), and strongly in \( L^2_{\rW}(B_R) \), as well as Lebesgue-a.e., as \( k \to +\infty \) for each \( R > 0 \) (after passing to a subsequence if necessary).  

From this, it follows that \( \chi_{\{v > 0\}} \leq \liminf_{k \to +\infty} \chi_{\{v_k > 0\}} \) Lebesgue-a.e. Then, applying Fatou’s lemma, we obtain  
\begin{align*}
\frac{\left|B_1 \cap \{v > 0\}   \right|}{|B_1|}
\leq \liminf_{k \to +\infty} \frac{\left |B_1 \cap \{v_k > 0\}  \right|}{|B_1|}
= \liminf_{k \to +\infty}\frac{\left| B_{r_k}(X_0) \cap \{\psi> 0\}  \right|}{|B_{r_k}(X_0)|}.
\end{align*}  
Since \( |X\cdot\nu| > 0 \) Lebesgue-a.e., property \textit{(2)} follows directly from \textit{(1)}.

Due to the existence of an approximate tangent for \( \boldsymbol{\nu} \), for any \( B_{\tau} (X) \) with \( X \in L \), we have  
\begin{align*}
\liminf_{k \to + \infty} \boldsymbol{\nu}_{X_0,r_k}
(B_{\tau} (X)) \geq \mathcal{H}^{1}
(B_{\tau} (X) \cap L) \geq 2 \tau.
\end{align*}  
In particular, for sufficiently large \( k \), the set \( B_{\tau r_k} (X_0 + r_k X) \cap \Sigma_H \) is non-empty. 
Choosing the unit vector $\nu$ such that \( L = \{X\cdot \nu = 0\} \) and let $\boldsymbol{n}$ be the unit vector perpendicular to $\nu$.
We can find points \( X_{k} \in  B_{r_k/40}(X_0 + r_k \boldsymbol{n}) \cap \Sigma_H\); let $X_k := X_0  + r_k \tilde X_k$.  

Using Lemmas \ref{42-N} and \ref{46}, we know that  
\begin{align*}
\frac{1}{C} H_{X_0}(\psi, r_k) \leq R^{-3} r_k^{3} H_{X_k}(\psi, Rr_k) \leq C H_{X_0}(\psi, r_k),
\end{align*}  
for every \( R \in [1/40, 40] \) and for some constant \( C \) independent of \( r_k \) and \( X_{k} \). Then, from the formulas in Remark \ref{r2} and Remark \ref{rmrk-almost-almgren}, we obtain:  
\begin{align*}
\int_{r_k/4}^{4r_k} \int_{\partial B_s(X_{k})} &\frac{2s^{-1}}{H_{X_0}(\psi,r_k)} \frac{1}{y} \left|\nabla \psi(X) \cdot (X-X_k) - \cN_{X_{k}}(\psi,s) \psi(X) \right|^2\,d\mathcal{H}^{1}\, ds \\
&\leq  (y_k-4r_k)^{-2} \cN_{X_{k}}(\psi,4r_k)-(y_k-r_k/4)^{-2}\cN_{X_{k}}(\psi, r_k/4) + C r_k^{1/2} \leq C \eta + C r_k^{1/2},
\end{align*}  
using \( \cN_{X_{k}}(r) \in [1-C_\star r^{1/2}, 1 + \eta] \) for sufficiently large \( k \). Applying this bound again, we find  
\begin{align*}
\int_{B_4(\tilde X_{k})\setminus B_{1/4}(\tilde X_{k})}
\frac{1}{y_0+r_ky}\left|\nabla v_k(X) \cdot (X-\tilde X_k) - v_k(X) \right|^2 \, dX \leq C\eta+ C r_k^{1/2}.
\end{align*}  
The same estimate holds at \( X_0 \):  
\begin{align*}
\int_{B_4\setminus B_{1/4}}
\frac{1}{y_0+r_ky}\left|\nabla v_k(X) \cdot X - v_k(X) \right|^2\,dX \leq C\eta+ C r_k^{1/2}.
\end{align*}  
Applying Lemma \ref{6.2}, we obtain  
\begin{align*}
\int_{B^{+}_2(X_0)\setminus B^{+}_{1/2}(X_0)}
\frac{1}{y_0+r_ky}\left|\nabla v_k(X) \cdot \tilde X_{k} \right|^2 \,dX \leq C\eta+ C r_k^{1/2}.
\end{align*}  
This implies  
\begin{align*}
\int_{B_2 \setminus B_{1/2}} \frac{1}{y_0+r_ky}\left|\nabla v_k \cdot \boldsymbol{n} \right|^2 \,dX\leq C\eta+Cr_k^{1/2},
\end{align*}  
due to the choice of \( X_{k} \). 
Since \( \eta \) was arbitrary, we conclude  
\begin{align*}
\lim_{k \to + \infty} \int_{B_2\setminus B_{1/2}} \frac{1}{y_0+r_ky}\left|\nabla v_k \cdot \boldsymbol{n}  \right|^2 \,dX=0.
\end{align*}  
Applying Theorem \ref{5.3}, we deduce that \( v(X) = \alpha(X_0) |X\cdot \nu| \) on \( B_1 \setminus \overline{B_{1/2}} \). 
A rescaled version of the same argument extends this conclusion to any annulus \( B_R \setminus \overline{B_{R/2}} \), and thus to all of \( \mathbb{R}^2 \). 

To establish the strong convergence \( v_k \to v \), consider  \( \eta \in C^{\infty}_c(\mathbb{R}^2) \), we have  
\begin{align*}
\int \frac{\eta}{y_0} |\nabla v|^2 \, dX &= -2\int \frac{v}{y_0} \nabla \eta \cdot \nabla v \, dX = \lim_{k \to + \infty} -2 \int \frac{v_k}{y_0+r_ky} \nabla \eta \cdot \nabla v_k \, dX = \lim_{k \to + \infty} \int \frac{\eta}{y_0+r_ky} |\nabla v_k|^2 \, dX,
\end{align*}  
where the first and last terms follow from the fact that \( \Delta v =0 \) in $\{v>0\}$,  and that each \( v_k \) satisfies
 $\mathrm{div\,} \left(\frac{1}{y_0+r_ky}\nabla v_k \right) =0$ in $\{v_k>0\}$.
Since the limit function \( \alpha(X_0) |X\cdot\nu| \) is independent of the chosen subsequence, we conclude that \textit{(1)} holds.
\end{proof}

\subsection{Identification of the limit \texorpdfstring{$I$}{I}}\label{8.3}

\begin{theorem}\label{86}
Let $(\psi, I)$ be a variational solution in an open and connected domain $\Omega$. Then, either $\psi \equiv 0$ or $I = \chi_{\{\psi>0\}}$ Lebesgue-a.e.
\end{theorem}

\begin{proof}
From Lemma \ref{82} and Theorem \ref{72}, we obtain that $|\partial \{\psi> 0\}| = 0$. Moreover, by definition, on $\{\psi> 0\}$ we have $I(X) = 1 = \chi_{\{\psi>0\}}(X)$ Lebesgue-a.e. Therefore, it suffices to show that $I = 0$ Lebesgue-a.e. in the interior $\{\psi = 0\}^{\circ}$.

Consider a non-empty connected component $U$ of $\{\psi = 0\}^{\circ}$ on which $I = 1$, recalling that $I$ is locally constant on $\{\psi = 0\}^{\circ}$. Clearly, we have $\partial U \cap \Omega \subseteq \partial \{\psi> 0\}$. For any $X_0 \in \partial U$, it follows from Proposition \ref{Highest-frequency-V} and the upper semi-continuity of $X_0 \mapsto M_{X_0}(0^+)$  that  
\[
M_{X_0}(0^+) \geq \limsup_{U \ni X\to X_0 } M_X(0^+) = \omega(X_0).
\]
In particular, we have $\partial U \cap \Omega \subseteq \Sigma_H$. 
By applying Theorem \ref{72}, we conclude that $\partial U \cap \Omega$ has locally finite Hausdorff measure, implying that $U$ has locally finite perimeter.

Assume that $U \neq \Omega$. Then, there exists a smooth, connected, open set $\Omega' \Subset \Omega$ such that $ |U \cap \Omega'| > 0$ and $|\Omega'\setminus U| > 0$. 
By the relative isoperimetric inequality, we obtain $\mathcal{H}^{1}\left(\partial^{\star} U \cap \Omega'\right) > 0$, where $\partial^{\star}U$ denotes the reduced boundary.
At each point $X_0 \in \partial^{\star}U$, the set $U$ has Lebesgue density $1/2$, so $\{\psi> 0\} \subseteq \Omega \setminus U$ can have at most Lebesgue density $1/2$ at point $X_0$. 
On the other hand, by Lemma \ref{85}, the Lebesgue density of $\{\psi> 0\}$ at $\mathcal{H}^{1}$-a.e. point in $\Sigma_H$, which contains $\partial^{\star}U \subseteq \partial U$, is $1$. This leads to a contradiction.
\end{proof}

\subsection{Identification of the measure \texorpdfstring{$\mathcal{L}\psi$}{Lpsi}}
\label{8.4}

\begin{theorem}
\label{87}
Let $(\psi, I)$ be a variational solution on $\Omega$. Then, the set $\partial \{\psi> 0\}$ is countably $\mathcal{H}^{1}$-rectifiable, has locally finite $\mathcal{H}^{1}$-measure, and satisfies in the sense of distributions  
\begin{align*}
\mathcal{L}\psi =  \mathcal{H}^{1}\llcorner \partial^{\star}\{ \psi> 0\}+ \frac{2}{\sqrt{\pi y}} \sqrt{\lim_{r \to 0^+} r^{-3} H_{X}(\psi, r)}\mathcal{H}^{1}\llcorner \Sigma_H.
\end{align*}
\end{theorem}

\begin{proof}
The measure $\mathcal{L}\psi$ is supported on $\partial \{\psi> 0\}$, and applying Corollary \ref{84} to $\psi$ on the open set $\Omega \setminus \Sigma_H$, yields:
\begin{align*}
\mathcal{L}\psi \llcorner(\Omega \setminus \Sigma_H) =  \mathcal{H}^{1} \llcorner \partial^{\star}\{\psi> 0 \}. 
\end{align*}
Additionally, from Lemma \ref{BV-estimate}, we have the estimate:
\begin{align*}
\mathcal{L}\psi \left(B^{+}_r(X_0) \right) \leq C(C_V)r
\end{align*}
when $B^{+}_{2r}(X_0) \Subset \Omega$. This implies that:  
$$ \mathcal{H}^{1} \left( \left(\partial \{\psi> 0\} \setminus \Sigma_H \right) \cap \Omega' \right) = \mathcal{H}^{1} \left(\partial^{\star}\{\psi> 0\}\cap \Omega' \right) < +\infty, $$
for any $\Omega' \Subset \Omega$. On the other hand, by Theorem \ref{72}, we have $\mathcal{H}^{1} \left (\Omega' \cap \Sigma_H \right) < +\infty$, which leads to  
\[
\mathcal{H}^{1} \left(\Omega' \cap \partial \{\psi>0\} \right) < +\infty.
\]
Moreover, since $\Sigma_H$ is countably $\mathcal{H}^{1}$-rectifiable, it follows that $\partial \{\psi> 0\}$ is also countably $\mathcal{H}^{1}$-rectifiable.

By \cite[3.2.19]{zbMATH03280855} and the Radon-Nikodym theorem, if we define the density function
\begin{align*}
\theta(X_0) = \limsup_{r \to 0^+}\frac{\mathcal{L}\psi \left(B_r(X_0)\right)}{2r} \leq C(C_V),
\end{align*}
then
\begin{align*}
\mathcal{L}\psi= \theta \mathcal{H}^{1}\llcorner \partial \{\psi>0\}.
\end{align*}
Thus, it suffices to show that for $\mathcal{H}^{1}$-a.e. point $X_0=(x_0,y_0) \in  \Sigma_H$, we have:
\begin{align*}
\theta(X_0)=\frac{2}{\sqrt{\pi y_0}}\sqrt{\lim_{r \to 0^+} r^{-3} H_{X_0}(\psi, r)}.
\end{align*}
Applying Lemma \ref{85}, for $\mathcal{H}^{1}$-a.e. point $ X_0 \in \Sigma_H$, we have that  
\[
v_{X_0,r}(X) = \frac{\psi(X_0+rX)}{\sqrt{r^{-1} H_{X_0}(\psi,r)}} \to v(X) = \alpha(X_0) |X\cdot \nu| = \sqrt{\frac{y_0}{\pi}} |X\cdot \nu|
\]
in $W^{1,2}_{\rW,\mathrm{loc}} (\mathbb{R}^2)$, for some unit vector $\nu$. 
Define the rescaled operator $\cL_r v = \mathrm{div\,} \left(\frac{1}{y_0+ry}\nabla v \right)$, then  
\begin{align*}
\frac{\mathcal{L}\psi\left(B_r(X_0)\right)}{ \sqrt{r^{-1} H_{X_0}(\psi,r)}}=\mathcal{L}_rv_{X_0,r}(B_1) \to \frac{1}{y_0} \Delta v(B_1) = \frac{4}{\sqrt{\pi y_0}}
\end{align*}
as $r \to 0$, using the fact that $\Delta v(\partial B_1) =0$ and that $\mathcal{L}_rv_{X_0,r} \to \frac{1}{y_0} \Delta v$ in the weak-* sense of measures, as $r \to 0$, while explicitly computing the Laplacian of $v$ in the final step. Finally, 
\begin{align*}
\theta(X_0) = \limsup_{r \to 0^+}\frac{\mathcal{L}\psi \left(B_r(X_0)\right)}{2r} 
= \frac{2}{\sqrt{\pi y_0}} 
\sqrt{\lim_{r \to 0^+} r^{-3} H_{X_0}(\psi, r)},
\end{align*}
which completes the proof.
\end{proof}

\begin{proof}[Proof of the Main Theorem \ref{MAIN-THEOREM}]

(i) is a consequence of Theorem \ref{86}.

(ii) has been established in Theorem \ref{87}.

(iii) follows from Theorems \ref{87} and \ref{85}.

(iv) follows from Lemma \ref{85} and Theorem \ref{72}.
\end{proof}

\paragraph{\bf{Acknowledgements}}
M. Fotouhi and P. Vosooqnejad were supported by Iran National Science Foundation (INSF) under project No. 4001885. M. Bayrami was supported by a grant from IPM. This research was in part supported by a grant from IPM (No.1404350111).

\section*{Declarations}

\medskip
\noindent {\bf  Funding and/or Conflicts of interests/Competing interests:} The authors declare that there are no financial, competing or conflict of interests.


\bibliographystyle{acm}
\bibliography{mybibfile}

\end{document}